\documentclass[11pt]{amsart}
\usepackage[greek,english]{babel}
\usepackage[utf8x]{inputenc}
\usepackage[T1]{fontenc}
\usepackage{a4wide}
\usepackage{mathpazo}
\usepackage[dvipsnames]{xcolor}
\usepackage[colorlinks=true,allcolors=MidnightBlue,linktoc=page]{hyperref}

\usepackage{mathrsfs}
\usepackage{amssymb, amsmath}
\usepackage{enumitem}
\usepackage{graphicx}
\usepackage{pgf,tikz}
\usetikzlibrary{cd}
\usetikzlibrary{shadings,intersections}
\definecolor{mediumtealblue}{rgb}{0.0, 0.33, 0.71}
\usepackage{mathrsfs}
\usetikzlibrary{arrows} 
\usepackage{xspace}
\usepackage[all]{xy}
\usepackage{bm}
\usepackage{scalerel,stackengine}
\stackMath
\newcommand\reallywidehat[1]{%
\savestack{\tmpbox}{\stretchto{%
  \scaleto{%
    \scalerel*[\widthof{\ensuremath{#1}}]{\kern-.6pt\bigwedge\kern-.6pt}%
    {\rule[-\textheight/2]{1ex}{\textheight}}
  }{\textheight}%
}{0.5ex}}%
\stackon[1pt]{#1}{\tmpbox}%
}

\newcommand{\D}{\mathbf D}
\newcommand{\F}{\mathbf F}
\newcommand{\Z}{\mathbf Z}

\newcommand{\NN}{\mathbf{N}}

\newcommand{\RR}{\mathbf{R}}
\newcommand{\R}{\mathbf{R}}

\newcommand{\QQ}{\mathbf{Q}}

\renewcommand{\H}{\mathbf{H}}

\newcommand{\cat}{{\upshape CAT($0$)}\xspace}

\newcommand{\teta}{\vartheta}

\newcommand{\action}{\curvearrowright}

\DeclareMathOperator{\Isom}{Isom}
\DeclareMathOperator{\Aut}{Aut}

\DeclareMathOperator{\im}{im}

\DeclareMathOperator{\Prob}{Prob}

\DeclareMathOperator{\Id}{Id}

\DeclareMathOperator{\Stab}{Stab}

\DeclareMathOperator{\OO}{O}
\DeclareMathOperator{\PO}{PO}
\DeclareMathOperator{\PP}{P}

\DeclareMathOperator{\GL}{GL}
\let\L\relax
\DeclareMathOperator{\L}{L}

\DeclareMathOperator{\Sa}{Sam}
\theoremstyle{plain}
\newtheorem{thm}{Theorem}[section]
\newtheorem*{thm*}{Theorem}
\newtheorem{lem}[thm]{Lemma}
\newtheorem{prop}[thm]{Proposition}
\newtheorem{cor}[thm]{Corollary}
\theoremstyle{definition}
\newtheorem*{defn*}{Definition}
\newtheorem{defn}[thm]{Definition}

\newtheorem*{example*}{Example}
\newtheorem{rem}[thm]{Remark}

\newtheorem*{rem*}{Remark}
\begin{document}

\begin{abstract}We consider the isometry group of the infinite dimensional separable hyperbolic space  with its Polish topology. This topology is given by pointwise convergence. For non-locally compact Polish groups, some striking phenomena like automatic continuity or extreme amenability may happen. Our leading idea is to compare this topological group with usual Lie groups on one side and with non-Archimedean infinite dimensional groups like $\mathcal{S}_\infty$, the group of all permutations of a countable set on the other side. Our main results are \\

\begin{itemize}
\item Automatic continuity (any homomorphism to a separable group is continuous).
\item Minimality of the Polish topology.
\item Identification of its universal Furstenberg boundary as the closed unit ball of a separable Hilbert space with its weak topology.
\item Identification of its universal minimal flow as the completion of some suspension of the action of the additive group of the reals $\R$ on its universal minimal flow.\\
\end{itemize}

All along the text, we lead a parallel study with the sibling group $\Isom(\mathcal{H})$, where $\mathcal{H}$ is a separable Hilbert space. 
\end{abstract}

\title[Isometries of infinite dimensional hyperbolic spaces]{The Polish topology of the isometry group of the infinite dimensional hyperbolic space}
\author[B. Duchesne]{Bruno Duchesne}
\address{Université de Lorraine, CNRS, IECL, F-54000 Nancy, France}
\thanks{B.D. is supported in part by French project ANR-16-CE40-0022-01 AGIRA}
\date{January 2023}
\maketitle

\begin{flushright}
\begin{minipage}[t]{0.7\linewidth}\itshape\small
La plaisanterie, qui n’en était pas entièrement une, fut prise au sérieux et tous les experts ont depuis adopté cette étrange terminologie...
\begin{flushright}\upshape\small
--- Roger Godement, \textit{Analyse mathématique IV.}\\
\end{flushright}
\end{minipage}
\end{flushright}
\tableofcontents
\section{Introduction}
\subsection{The Polish topology on the hyperbolic isometries} Alike Euclidean spaces, hyperbolic spaces have natural infinite dimensional analogues. Here we consider the separable one. Let $\mathcal{H}$ be some separable Hilbert space with some  Hilbert base $(e_i)_{i\in\NN}$. Let $Q$ be the quadratic form with signature $(1,\infty)$ defined by 
$$Q(x)=x_0^2-\sum_{i\geq1}x_i^2$$

where $(x_i)$ are the coordinates of $x$ in the above base. We denote by $\OO(1,\infty)$ the orthogonal group of $Q$. The infinite dimensional (separable) hyperbolic space is then 
$$\H=\left\{x\in\mathcal{H};\ Q(x)=1,\ x_0>0\right\}.$$

The hyperbolic metric can be defined via the formula $$\cosh(d(x,y))=\left( x,y\right)=\langle x,Jy\rangle$$ where $\left(\cdot,\cdot\right)$ is the symmetric bilinear form obtained by polarization from $Q$, $J$ is the linear operator such that $Je_0=e_0$ and $Je_i=-e_i$ for $i>0$ and $\langle\cdot,\cdot\rangle$ is the scalar product on $\mathcal{H}$.

This space appears in several previous works. Among some examples, it appears in \cite{MR2152540, MR2811600,  MR2881312, MR3558533, MR3263898, MR3978389, MR4121370}. The group of isometries of $\H$ is $$\Isom(\H)=\PO(1,\infty)=\OO(1,\infty)/\{\pm\Id\}.$$
A proof of this identification can be obtained by mimicking the classical proof in finite dimension \cite[Theorem I.2.24]{MR1744486}. It is also a particular case of identification of isometry groups of infinite dimensional Riemannian symmetric spaces. See \cite[Theorem 1.5]{MR3044451} and \cite[Theorem 3.3]{duchesne2019representations}.

This infinite dimensional group has a natural  group topology which makes the action $\Isom(\H)\times \H\to \H$ continuous. This is the \emph{pointwise convergence} topology, that is the coarsest group topology on $\Isom(\H)$ such that orbit maps $g\mapsto gx$ are continuous. Actually, this topology is merely the quotient topology of the strong operator topology on $\OO(1,\infty)$ (Proposition~\ref{strong}).  Since $\H$ is separable, this topology is known to be Polish \cite[\S9.B]{MR1321597}, that is separable and completely metrizable.

The aim of this paper is to study this Polish group. It lies at the crossroads of two worlds: classical Lie groups and non-locally compact Polish groups with surprising properties for group theorists used to the locally compact ones. Actually, $\OO(1,\infty)$ \emph{is} a Lie group --more precisely a Banach-Lie group-- but for the stronger topology given by the operator norm.

There is another Polish group which really looks like $\Isom(\H)$. This is the isometry group of the Hilbert space $\Isom(\mathcal{H})$. Actually, these two groups are homeomorphic but the homeomorphism (provided by the Cartan decomposition proved in Proposition~\ref{Cartan}) is not a group homomorphism.   So, as a leitmotif, we will compare $\Isom(\H)$ and $\Isom(\mathcal{H})$ all along the paper.

 \subsection{Flows and Amenability} Actions on compact spaces are useful tools to study topological groups. For example, they play a crucial role in Furstenberg’s boundary theory and rigidity results for lattices of Lie groups. A \emph{flow} or a $G$-\emph{flow} is a continuous action $G\times X\to X$ of a topological group $G$ on a non-empty compact Hausdorff space $X$. A flow is \emph{minimal} if every orbit is dense, equivalently, there is no non-trivial closed invariant subspace. One can reduce to the study of minimal flows since any flow contains a minimal one. Ellis proved that each topological group $G$ has a universal minimal flow $M(G)$ in the sense that for any other flow $X$, there is a continuous $G$-equivariant map $M(G)\to X$. In some sense, the study of all minimal flows is contained  in the study of the universal one. See for example \cite{MR0474243} for a general reference.
 
 For locally compact non-compact groups $G$, this universal minimal flow $M(G)$ is very large (it is non-metrizable for example) and we lack a handy description. On the other hand some infinite dimensional\footnote{We use the terms "infinite dimensional groups" as a synonym of non-locally compact groups, as it is customary.} Polish groups have an easily describable universal minimal flow: It is reduced to a point. Such groups are called \emph{extremely amenable,} and equivalently any continuous action on a compact space has a fixed point. One of the first groups known to be extremely amenable is the orthogonal group $O$ of  a separable Hilbert space. It appears as a consequence of the concentration of the measure phenomenon for Euclidean spheres in high dimensions \cite{MR708367,MR1900705}.
 
 For simple non-compact Lie groups $G$, a flow is provided by the homogeneous space $G/P$ where $P$ is a minimal parabolic subgroup. This flow is often called the \emph{Furstenberg boundary} of $G$. It is, moreover, strongly proximal: any probability measure on $G/P$ has a Dirac mass in the closure of its orbit in the compact space of probability measures on $G/P$ with its weak-* topology. This flow is actually the universal  strongly proximal minimal flow of $G$: any other strongly proximal minimal flow is an equivariant continuous image of $G/P$ \cite[Chapter 6]{MR0474243}. 
 
 For finite dimensional hyperbolic spaces $\H^n$ and its isometry group $\PO(1,n)$, the parabolic subgroup $P$ is the stabilizer of a point at infinity and $G/P$ is the sphere at infinity $\partial \H^n$. In our infinite dimensional setting, the sphere at infinity $\partial \H$ is no more compact and thus does not provide a flow.
 
 To any metric space $(X,d)$, one can associate a compactification of $X$ which is the completion of the uniform structure induced by the functions $x\mapsto d(x,y)-d(x,z)$ from $X$ to $[-d(y,z),d(y,z)]$. This is the \emph{horicompactification} of $X$. For both $\H$ and $\mathcal{H}$, the points of this horicompactification have been described explicitly in \cite{MR3998735,MR4117572}. Here we give a slightly different presentation emphasizing the topological structure and we explicit how it is an $\Isom(X)$-flow. It is actually a closed subspace of a product space on which $\Isom(X)$ acts by permutation of the factors.
 
 A bit surprisingly, the horicompactifications of $\H$ and $\mathcal{H}$ are the same as topological spaces. Let $\mathbf{D}$ be the open unit ball in $\mathcal{H}$ and $\overline{\mathbf{D}}$ its closure endowed with the weak topology. The \emph{frustum} (or merely the truncated cone) over $\overline{\mathbf{D}}$ is $$\mathbf{F}=\{(x,r)\in\overline{\D}\times[0,1], \ ||x||\leq r\}.$$ 
 Observe that $\mathbf{F}$ is homeomorphic to the Hilbert cube thanks to Keller's theorem (see e.g. \cite[Theorem 12.37]{MR2766381}).
 \begin{thm} The horicompactifications of $\H$ and $\mathcal{H}$ are homeomorphic to the frustum $\mathbf{F}$.
 \end{thm}
 
 So $\mathbf{F}$ is a flow for both $\Isom(\H)$ and $\Isom(\mathcal{H})$ but there is an important difference between  these two flows. The projection map $\mathbf{F}\to\overline{\mathbf{D}}$ is a factor map between $G$-flows only in the case $G=\Isom(\H)$.
 
 The fact that the closed unit ball with the weak topology $\overline{\mathbf{D}}$ is a flow for $\Isom(\H)$ can be seen more geometrically using the Klein model. See Proposition~\ref{Gflow}. On the contrary, Pestov proved that $\Isom(\mathcal{H})$ is extremely amenable \cite[Theorem 6.15]{MR1900705} and since the action  $\Isom(\mathcal{H})\action\overline{\mathbf{D}}$ has no fixed point, it cannot be continuous.
 
 The action $\Isom(\H)\action\overline{\mathbf{D}}$ is, moreover, strongly proximal, minimal and this is indeed the universal strongly proximal minimal flow (or Furstenberg boundary) of $\Isom(\H)$.
 
 \begin{thm}\label{usp} The universal strongly proximal minimal flow of $\Isom(\H)$ is $\overline{\mathbf{D}}$.\end{thm}
 
 Let us emphasize that this flow has two orbits (the open ball and the unit sphere) and the sphere is comeager. Moreover, this is the universal proximal flow of $\Isom(\H)$ as well (Theorem~\ref{prox}). Once the action $\Isom(\H)\action\overline\D$ has been proved to be continuous, the key point of Theorem~\ref{usp} is the fact that stabilizers of points at infinity of $\H$ are amenable subgroups.

 The simplicity of this flow allows us to investigate the structure of the universal minimal flow $M(G)$ of $G=\Isom(\H)$. By universality, there is a continuous $G$-map $M(G)\to\overline{\mathbf{D}}$ and for $x\in\overline{\mathbf{D}}$, $M_x=\pi^{-1}(\{x\})$ is a $G_x$-flow where $G_x$ is the stabilizer of $x$.  So,  the preimage by $\pi$ of the orbit of $x$ is a suspension of the action $G_x\action M_x$. Since the sphere (identified with $\partial \H$ in the Klein model) is a comeager orbit, it is natural to try to understand what is the suspension over this sphere $\partial \H$.
 
 Let $\xi\in \partial \H$ and $G_\xi$ be its stabilizer. One has the following short exact sequence

 $$0\to H_\xi\to G_\xi\to\R\to0$$
 
 where $H_\xi$ is the kernel of the Busemann homomorphism $\beta_\xi\colon G_\xi\to\R$. This homomorphism is defined int the following way. If $\xi\in\partial \H$ and $x,y$, the Busemann function at $(x,y)$ is $b_\xi(x,y)=\lim_{t\to\infty}d(\rho(n),x)-d(\rho(n),y)$ where $\rho\colon\R_+\to X$ is a geodesic ray with $\rho(0)=y$ and $\lim_{n\to\infty}\rho(n)=\xi$. It satisfies the following cocycle formula: $b_\xi(x,y)+b_\xi(y,z)=b_\xi(x,z)$. On $G_\xi$, the map $g\mapsto\beta_\xi(g)=b_\xi(gx,x)$ is a homomorphism with values independent of the choice of $x\in \H$. See \cite[Chapter II.8]{MR1744486} for details about Busemann functions.\\
 
 Actually, the group $H_\xi$ is isomorphic to $\Isom(\mathcal{H})$ and so is extremely amenable as well (Lemma~\ref{isomext}). In particular, for any minimal $G_\xi$-flow $M$, $H_\xi$ acts trivially and thus $M$ is an $\R$-flow. Let $M(\R)$ be the universal minimal flow of the reals. Let us denote by $S(M(\R))$ the completion of the suspension of the action $G_\xi\action M(\R)$ with respect to some natural uniform structure. This construction is detailed in  Section \ref{umf}. Essentially, it is obtained in the following way: the group $G$ acts on $M(\R)\times G/G_\xi$ via
 
 $$g(m,\eta)=(m+c(g,\eta),g\eta)$$
  where $c$ is a cocycle constructed from Busemann functions. The space is not complete (as uniform space), we complete it and show that the action extends continuously.
 \begin{thm} The universal minimal flow of $\Isom(\H)$ is the completed suspension $S(M(\R))$.\end{thm}
 
 As a corollary, we get that this flow is not metrizable, as it is the case in finite dimension. For infinite dimensional Polish groups, universal minimal flows are often obtained as Samuel compactifications of homogeneous space $G/H$ where $H$ is an extremely amenable closed subgroup. For $\Isom(\H)$, maximal extremely amenable subgroups are stabilizers of points $G_x$ with $x\in\H$ and horospherical subgroups $H_\xi$ with $\xi\in\partial \H$. In both cases, the Samuel compactification $\Sa(G/H)$ is not minimal (Lemma \ref{not_min}) and thus the universal minimal flow of $\Isom(\H)$ is not of the form $\Sa(G/H)$. So, whereas the Furstenberg boundary is covered by the quite common situation described in \cite[Theorem 7.5]{MR4266370} for universal flows with a comeager orbit, this is not the case for the universal minimal flow.
 
\subsection{Automatic Continuity}Another surprising property that some Polish groups may have is automatic continuity. A topological group $G$ has  \emph{automatic continuity} if any homomorphism $G\to H$ where $H$ a separable Hausdorff topological group is continuous. 

In \cite{MR3978389}, Monod and Py proved that irreducible self-representations of $\Isom(\H)$ are automatically continuous and asked more generally whether automatic continuity holds for $\Isom(\H)$ in  \cite[\S1.3]{MR3978389}.

We answer positively this question for the groups $\Isom(\H)$ and $\Isom(\mathcal{H)}$. As it is well known, automatic continuity implies uniqueness of the Polish topology and thus we can speak about the Polish topology of either group.

\begin{thm}\label{autcont} The groups $\Isom(\H)$ and $\Isom(\mathcal{H})$ have automatic continuity.\end{thm}

A common strategy to prove automatic continuity is to prove the existence of ample generics \cite{MR2535429}. Here, this is not possible since $\Isom(\H)$ has no dense conjugacy class (Theorem~\ref{nodenseconj}). With respect to this property, the group $\Isom(\H)$ looks like its finite dimensional siblings.

We prove that both $\Isom(\H)$ and $\Isom(\mathcal{H})$ have the Steinhaus property and this property is well known to imply automatic continuity. Let $G$ be a topological group. A subset $W\subset G$ is $\sigma$-syndetic if $G$ is the union of countably many left translates of $W$. It is symmetric if $W=W^{-1}=\{w^{-1},\ w\in W\}$. The group $G$ has the \emph{Steinhaus property} if there is some natural integer $k$ such that for any $\sigma$-syndetic symmetric subset $W\subset G$, $W^k$ contains an open neighborhood of the identity. To prove this property for  $\Isom(\H)$ and $\Isom(\mathcal{H})$, we rely on the same property for the orthogonal group, which was proved by Tsankov \cite{MR3080189}. The orthogonal group appears as a point stabilizer in both groups.

While the groups $\Isom(\H)$ and $\Isom(\mathcal{H})$ exhibit strong geometric differences, the proof of the Steinhaus property is the same for both groups and rely on the use of the stabilizers of three non-aligned points. 

\begin{rem}Remark In \cite{MR3936642} (and its corrigendum \cite{MR4515296}), Sabok proved a very general automatic continuity result for isometry groups. It’s very likely that our result is a consequence of this general result  that uses continuous model theory. We preferred to use a shorter path relying directly on the Steinhauss property for the orthogonal group due to Tsankov. This later result can also be retrieved as a consequence of Sabok’s result.
\end{rem}

\subsection{Minimality of the topology}  Automatic continuity means that the Polish topology is maximal (i.e., the finest) among separable group topologies on $G$. In the other direction, one can look for minimality properties for $G$. A Hausdorff topological group $(G,\tau)$ is said to be \emph{minimal} if there is no Hausdorff group topology on $G$ coarser than $\tau$. Since the 1970s, minimal topological groups have been an active subject and we refer to the survey \cite{MR3205486} for background and history. For example, connected semisimple groups with finite center as well unitary or orthogonal groups of separable Hilbert spaces are minimal groups.

 Stojanov \cite{Stojanov_1984} gave the first proof than the orthogonal group $O$ of the separable Hilbert space $\mathcal{H}$ is minimal with the strong operator topology. This group can be thought as the isometry group of the unit sphere (or the projective space) with the angular metric.

The group $\Isom(\H)$ can be identified with the group of Möbius transformations of the unit sphere \cite{MR3978389} that are transformations that preserve angles infinitesimally. Using some ideas borrowed to Stojanov, we prove the minimality of $\Isom(\H)$. 

\begin{thm}\label{topmin} The Polish group $\Isom(\H)$ is minimal.\end{thm}

Since $\Isom(\H)$ is topologically simple, we get immediately that $\Isom(\H)$ is \emph{totally minimal}, that is any continuous homomorphism to a Hausdorff topological group is open. Combining minimality and automatic continuity, we get, moreover, the following characterization of the Polish topology.

\begin{cor} The Polish topology on $\Isom(\H)$ is the unique separable Hausdorff group topology on  $\Isom(\H)$.\end{cor}

For the isometry group of the Hilbert space, the group structure is different since the orthogonal group $O$ is a quotient of $\Isom(\mathcal{H})$ but we prove  minimality of the topology in a different way.

\begin{thm}\label{topmineuc} The Polish group $\Isom(\mathcal{H})$ is minimal.\end{thm}

As above we get the uniqueness of the separable  Hausdorff topology on $\Isom(\mathcal{H})$. Let us observe that this answers the question whether the isometry group of a separable homogeneous complete metric $X$ is minimal \cite[Question 4.33]{MR3205486} in the cases $X=\mathcal{H},\H$.

\subsection{Dense conjugacy classes} For matrix groups, the spectrum is a continuous conjugacy invariant and this fact essentially proves that there is no dense conjugacy class. Actually, Wesolek proved that no locally compact second countable group  has a dense conjugacy class \cite{MR3415606}.

In infinite dimension, some Polish groups like $\mathcal{S}_\infty,\Aut(\QQ,<), \Aut(\mathcal{R})$ where $\mathcal{R}$ is the random graph or $\Aut(D_\infty)$ where $D_\infty$ is the universal Wa\.zewski dendrite \cite{MR4077581} have dense conjugacy classes and even a comeager one.

The groups $\Isom(\H)$ and $\Isom(\mathcal{H})$ are in opposite position with respect to the existence of dense conjugacy classes.

\begin{thm} The Polish group $\Isom(\H)$ has no dense conjugacy class.
\end{thm}

\begin{thm} The Polish group $\Isom(\mathcal{H})$ has dense conjugacy classes.
\end{thm}

The explanation of this difference has a pure geometric origin. In a Euclidean space, bounded parts of hyperplanes can be approximated uniformly  by subsets of spheres with very large radii and this is not possible in hyperbolic spaces. Another way to express this approximation property is the fact that Euclidean hyperplanes coincide with horospheres whereas hyperplanes and horospheres are different things in hyperbolic spaces.

%
%
%
%
%
%


\subsection{Acknowledgements} This paper greatly benefits from discussions with Lionel Nguyen Van Thé. Nicolas Monod shared his ideas about the shape of the universal minimal flow and this was a great help for the description obtained here. Todor Tsankov made comments about a previous version of this paper and asked whether the topology of $\Isom(\mathcal{H})$ is minimal. This led to Theorem~\ref{topmin}. Michael Megrelishvili made important comments that helped me to improve this paper. The anonymous referee provided reference \cite{zbMATH03469702}. I would like to thank them warmly.

\section{The infinite dimensional hyperbolic space and its boundary at infinity}
\subsection{The hyperboloid and projective models}

There is another model of the hyperbolic space, close to the hyperboloid model described in the introduction. This is the projective model where $\H$ coincides with its image in the projective space $\PP(\mathcal{H})$. This way, one can identify $\H$ with lines in $\PP(\mathcal{H})$ that are positive for $Q$. In this model, the visual boundary $\partial \H$ is  given by $Q$-isotropic lines in $\PP(\mathcal{H)}$. We denote by $\overline{\H}$ the union $\H\cup\partial \H$. The \emph{cone topology} on $\overline{\H}$ is the one obtained from the inverse system of bounded convex subsets \cite[Chapter II.8]{MR1744486}. In finite dimension, it provides a compactification of $\H$ but this is no more the case in our infinite dimensional setting. In opposition to the weak topology described below, it can be thought as \emph{the strong topology} but it will have almost no role in this paper.

\subsection{The ball model or Klein model}\label{Klein} Let $\mathcal{H}_-$ be the closed subspace $\{x\in\mathcal{H},\ x_0=0\}$ and let $\D$ be the unit ball of $\mathcal{H}_-$. To an element $x\in \D$, we associate the point $f(x)=\frac{e_0+x}{\sqrt{1-||x||^2}}\in \H$. This map $f$ is a bijection between $\D$ and $\H$. It can be understood geometrically. The point $f(x)$ is the intersection of the line through the point $e_0+x$ and $\H$. The inverse map of $f$ can also be understood geometrically. If $y\in \H$, $f^{-1}(y)$ is given by the orthogonal projection on $\mathcal{H}_-$ of the intersection of  the line through $y$ and the affine hyperspace $\{x\in\mathcal{H},\ x_0=1\}$. The metric induced by the norm on $\D$ and the pullback of the hyperbolic metric via $f$ induced the same Polish topology but not the same uniform structure (The hyperbolic metric is complete whereas the norm is not complete on $\mathbf{D}$). When $\D$ is endowed with the pullback of the hyperbolic metric then geodesics are straight lines. This is the famous Klein model. In this case, the hyperbolic metric coincides with the Hilbert metric associated to the bounded convex subspace $\D\subset\H$. The identification between $\H$ and $\D$ induces an action of $\Isom(\H)$ on $\D$.


\begin{center}
\begin{figure}
\begin{tikzpicture}[scale=3]

  \coordinate (O) at (0,0);
  
    \draw[fill = black!30] (0,0.725) ellipse ({0.71} and {0.15});
  \draw[fill = black!30] (0,0) ellipse ({0.71} and {0.15});	
  \draw (0.725,0) node[right] {$\mathbf{D}$};
  
  \draw[densely dashed] (O) -- (1.33,1.33);
  \draw[densely dashed] (O) -- (-1.33,1.33);

     \draw[->] (O) -- (0,0.725);
     \draw (0,0.725) node[below right] {$e_0$};
     
   
 \draw[domain=-1.7:1.7] plot (0.725*\x,{ 0.725*sqrt(1+\x*\x)});
\fill[bottom color= black!20, top color=black!55] (1.24,1.46) -- (-1.24,1.46)-- plot [domain=-1.7:1.7] (0.725*\x,{ 0.725*sqrt(1+\x*\x)}) -- cycle;
  \shadedraw[bottom color= black!30, top color=black!60,color=black] (0,1.46) ellipse (12.4mm and  2.4mm);
  
   \draw (-1.29,1.52) arc [start angle=-200, end angle = 160,
    x radius = 13.75mm, y radius = 3.15mm];
    
    \filldraw (O) circle (.01)node[below]{$O$};
    
    \draw ( .5,-.7) node[above] {$\mathcal{H}_-$} ;
    \draw (-1.5,0) -- (.5,-.8)--(2,0);
    
    \coordinate (x) at (-.65,1.05);
   \path[name path = proj] (-.65,1.05)--(O);
    \path[name path = hori] (-1,0.725)--(1,0.725);
    \filldraw (x) circle (.01);
    \draw  (-.65,1.1) node[below left] {$f(x)$};
  \def\l{.64}
   \def\m{.57}
   
       \coordinate (p) at (\l*-.65,\l*1.05);
       \coordinate (m) at (\m*-.65,\m*1.05);

    \filldraw[name path =p] (\l*-.65,\l*1.05) circle (.01);
    \draw (\l*-.65,\l);
    \draw (-.65,1.05)--(\l*-.65,\l*1.05);
    \draw (\m*-.65,\m*1.05)--(O);
    \draw  (\l*-.65,\l*.94)--(\l*-.65,0) node[below] {$x$};
       \filldraw (\l*-.65,0) circle (.01);
    
\end{tikzpicture}
\caption{The correspondence between the hyperboloid model and the Klein model.}
\end{figure}
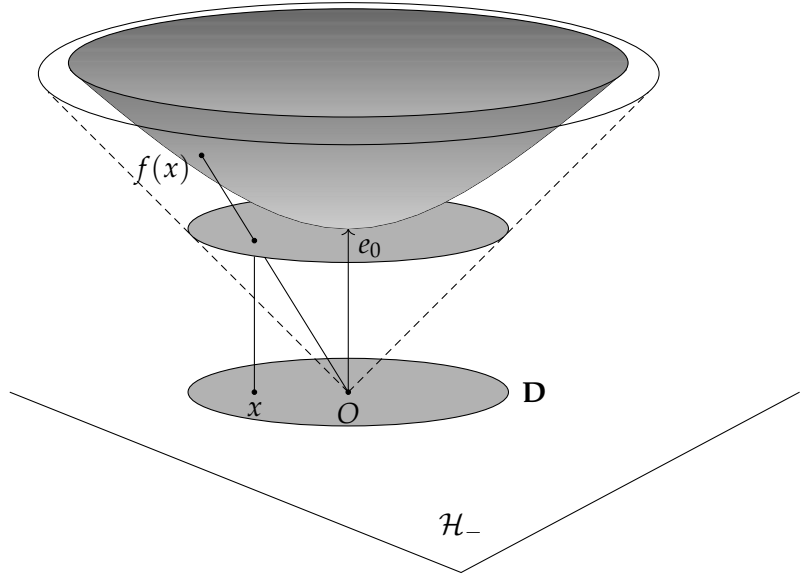
\end{center}

%

\subsection{The weak topology} A \emph{hyperplane} of $\overline\H$ is a non-empty  intersection of $\overline{\H}$ with a linear hyperplane of the Hilbert space $\mathcal{H}$. A \emph{closed half-space} of $\overline{\H}$ is the intersection of $\overline{\H}$ and a linear closed half-space. An \emph{open half-space} of $\overline{\H}$ is the complement of a closed half-space. Let us endow $\overline{\H}$ with the coarsest topology such that closed half-spaces are closed. This topology $\mathscr{T}_c$ is introduced in \cite[Example 19]{MR2219304}. It is proved that  $\overline{\H}$ is a compact Hausdorff space. 
In the ball model $\overline{\D}$, a closed half-space corresponds (via $f^{-1}$) to the intersection of $\overline{\D}$ with some affine closed half-space of $\mathcal{H}_-$ and thus the topology $\mathscr{T}_c$ coincides with the weak topology of this closed unit ball of the Hilbert space $\mathcal{H}_-$. In particular, it is metrizable. Thus we will call $\mathscr{T}_c$ the \emph{weak topology} on $\overline{\H}$. 

\begin{rem} Since $\H$ is a subspace of $\mathcal{H}$, one can also endow it with the restriction of the weak topology on $\mathcal{H}$. Let us denote by $\mathscr{T}'$ this topology.
Since $\cosh(d(x,y))=(x,y)=\langle x,Jy\rangle$, a sequence of $\H$ which is $\mathscr{T}'$-convergent in $\H$ is actually strongly convergent.
\end{rem}

\begin{lem} The collection of open half-spaces is a base of the weak topology on $\overline{\H}$.

\end{lem}

\begin{proof} By the definition of the $\mathscr{T}_c$-topology, it suffices to see that any finite intersection of open half-spaces contains some open half-space. Let us use the ball model. Let $U_1,\dots,U_n$ be open half-spaces (which are the intersection of open affine half-spaces of $\mathcal{H}_-$ with $\overline{\mathbf{D}}$) with non-empty intersection in $\overline{\D}$. Let $\xi$ be on the sphere $\partial \D$ and in the intersection $U_1\cap\dots\cap U_n$. Each $U_i$ has a boundary included in some closed affine hyperplane $H_i$. Let us denote $S_i=\partial \D\cap H_i$, which is a closed (for the strong topology) subspace of $\partial \D$ that does not contain $\xi$. In particular, there is $\alpha_i>0$, such that the ball of radius $\alpha_i$ for the angular metric on $\partial \D$ is included in $U_i$. Thus for $\alpha=\inf_i \alpha_i$, the ball of radius $\alpha$ around $\xi$ is included in  $U_1\cap\dots\cap U_n$. This spherical ball is exactly the intersection of some open half-space $U$ and $\partial \D$. In particular, $U$ is included in $U_1\cap\dots\cap U_n$.
\end{proof}

A bit counterintuitively, the substantial part of the closed unit ball, for the weak topology, is not the open unit ball but the unit sphere. So, stabilizers of points at infinity in $\Isom(\H)$ will play a more important role in the sequel than stabilizers of points in $\H$.

\begin{lem}\label{lem_com} The sphere at infinity $\partial \H$ is comeager in $\overline{\H}$ for the weak topology.\end{lem}
 
 \begin{proof} One can write $\H$ as the countable union of closed balls around $e_0$ with integral radius. None of these balls contains an open half-space and thus, they have empty interior. So $\H$ is meager.
 \end{proof}
 
 \begin{lem} The restriction of the weak topology on $\partial \H$ coincides with the cone topology.
 \end{lem}
 
 \begin{proof} We use the ball model. A sequence of unit vectors that converges weakly to a unit vector, converges strongly actually.
 \end{proof}
 
 \subsection{Horicompactifications} In this subsection, we recall a construction of a compact space associated to any metric space which is originally due to Gromov in a slightly different form. Let $(X,d)$ be some metric space with isometry group $G=\Isom(X)$ endowed with the pointwise convergence topology (recalled in Section \ref{secpolish}). For $x,y,z\in X$, let us define
 
 $$\varphi_{y,z}(x)=d(x,y)-d(x,z)\in[-d(y,z),d(y,z)]$$ and let $$X_2=\Pi_{y\neq z\in X}[-d(y,z),d(y,z)]$$ be the product space with the product topology. It is thus compact.
 
 The \emph{horicompactification} $\widehat{X}$ of $X$ is the closure of the image of $X$ in $X_2$ via the continuous map $$\varphi\colon x\mapsto\{\varphi_{y,z}(x)\}_{y\neq z}.$$
 
 Let us observe that for any $g\in G$, $x\neq y$ and $z\in X$, $\varphi_{y,z}(gx)=\varphi_{g^{-1}y,g^{-1}z}(x)$. In particular, the map $\varphi$ is equivariant with respect to the action of $G$ by permutations of the indices in $X_2$.
 
 
 \begin{lem}\label{horoflow} The map $\varphi\colon X\to\widehat{X}$ is an injective continuous map. The horicompactification $\widehat{X}$ is a $\Isom(X)$-flow which is metrizable as soon as $X$ is separable.\end{lem}
 
 \begin{proof} For an element $\psi\in X_2$, we denote by $\psi_{y,z}\in [-d(y,z),d(y,z)]$ its coordinate corresponding to $y\neq z\in H$. The continuity of $\varphi$ follows from the inequality 
 
 $$|\varphi_{y,z}(x)-\varphi_{y,z}(x’)|\leq2 d(x,x’)$$
 
 which shows actually that $\varphi$ is uniformly continuous. Injectivy of $\varphi$ follows from the fact that $|\varphi_{x,x’}(x)-\varphi_{x,x’}(x’)|=2d(x,x’)$. The triangle inequality implies that  for any $x,y,z\in X$,
 
\begin{equation}\label{lip}|\varphi_{y,z}(x)-\varphi_{y’,z’}(x)|\leq d(y,y’)+d(z,z’).\end{equation}

 In particular, for any $\psi\in\widehat{X}$, $|\psi_{y,z}-\psi_{y’,z’}|\leq d(y,y’)+d(z,z’)$. Thus, for any $\psi,\psi’\in\widehat{X}$, $g\in\Isom(X)$ and $y,z\in X$, 
 
 \begin{align*}
 |(g\psi')_{y,z}-\psi_{y,z}|&\leq|\psi’_{g^{-1}y,g^{-1}z}-\psi’_{y,z}|+|\psi’_{y,z}-\psi_{y,z}|\\
 &\leq d(gy,y)+d(gz,z)+|\psi’_{y,z}-\psi_{y,z}|.
 \end{align*}
 This shows the continuity of the action since the pointwise convergence topology on $\Isom(X)$ is a group topology. 
 
 The metrizability follows from Equation \eqref{lip} which shows that any $\psi\in\widehat{X}$ is completely determined by $(\psi_{y,z})_{y,z\in X_0}$ where $X_0$ is some dense countable subset of the metric space $X$.
 \end{proof}
 
 \begin{rem} This horicompactification is also known as the \emph{metric compactification} of $X$ \cite{MR2015055,MR3998735} and sometimes the space $\widehat{X}\setminus X$ is called the \emph{horofunction boundary}, for example in \cite{MR2456635}. 
 If we identify $X$ with its image in $X_2$, the induced topology from the product topology coincides with the topology $\mathscr{T}_w$ in \cite[\S3.7]{MR2219304}.
 
 Let us emphasize that Gromov originally defines a similar space in \cite{MR624814} but with the uniform convergence on bounded subsets. For proper metric spaces, this is equivalent with pointwise convergence but in our infinite dimensional context, the two notions of convergence are different.
 \end{rem}
 
  \begin{rem}\label{X_1} Often the horicompactification of a metric space is defined in a slightly different way. One fixes a point $x_0$ and consider the product $$X_1=\Pi_{y\neq x_0\in X}[-d(y,x_0),d(y,x_0)]$$ and the map 
 
 $$\begin{matrix}
 X&\to& X_1\\
 x&\mapsto&( \varphi_{y,x_0}(x))_{y\neq x_0}
 \end{matrix}$$
where $\varphi_{x,x_0}(y)=d(y,x)-d(y,x_0)$ as above.  The fact that $\varphi_{y,z}=\varphi_{y,x_0}-\varphi_{z,x_0}$ shows the closure of the images of $X$ in $X_1$ and $X_2$ are homeomorphic. A difference is the fact that $\varphi\colon X\to X_2$ is equivariant without the need to consider a quotient of the image. Moreover, the action on the image $\widehat{X}$ is maybe more explicit since it appears as a subshift of a generalized Bernoulli shift.
 \end{rem}
 
 For the separable Hilbert and hyperbolic spaces, explicit descriptions of the horofunction boundary are given in  \cite{MR3998735} and  \cite{MR4117572}. In the following two subsections we reformulate these descriptions to fit our objectives. We implicitly consider $\overline{\D}$ with its weak topology.
 \subsection{Horicompactification of the hyperbolic space}
 
 Let $\F$ be  the frustum $$\{(x,r)\in\overline{\D}\times[0,1], \ ||x||\leq r\}$$ considered as a subset of $\overline{\D}\times[0,1]$ with the topology product. This is a compact metrizable space with a continuous projection  $\pi\colon \F\to \overline{\D}$.
 
  \begin{figure}[h]
 \begin{center}
 \begin{tikzpicture}[scale=3]
  \coordinate (O) at (0,0);

    \def\rx{0.71}
    \def\ry{0.15}
    \def\z{0.725}
   \draw[fill = black!50] (O) ellipse ({\rx} and {\ry});

  \begin{scope}
    \path [name path = ellipse]    (0,\z) ellipse ({\rx} and {\ry});
    \path [name path = horizontal] (-\rx,\z-\ry*\ry/\z)
                                -- (\rx,\z-\ry*\ry/\z);
    \path [name intersections = {of = ellipse and horizontal}];

    \draw[bottom color= black!20, top color=black!70] (intersection-1) -- (0,0)
      -- (intersection-2) -- cycle;
    \draw[fill = black!60] (0,\z) ellipse ({\rx} and {\ry});
  \end{scope}

  \filldraw (O) circle (.01) node[below] {$0$};

  \draw (O) to (-0.71,0.71)  ;
  \draw (-0.71,0.5) node[left] {$\mathbf{F}$};
  \draw (O) -- (0.71,0.71);
  
  \draw[->] (0,\z)--(0,1) node[right] {$r$} ;
  \draw[dashed] (O)--(0,\z);
  \draw (\rx,-.1) node[right] {$\overline{\mathbf{D}}$};
  
    \draw ( .5,-.7) node[above] {$\mathcal{H}$} ;
    \draw (-1.5,0) -- (.5,-.8)--(2,0);
    
   \draw[->] (.8,\z)--(.8,0) node[midway,right] {$\pi$};
    
\end{tikzpicture}
\caption{The frustum $\mathbf{F}$.}
\end{center}
\end{figure}
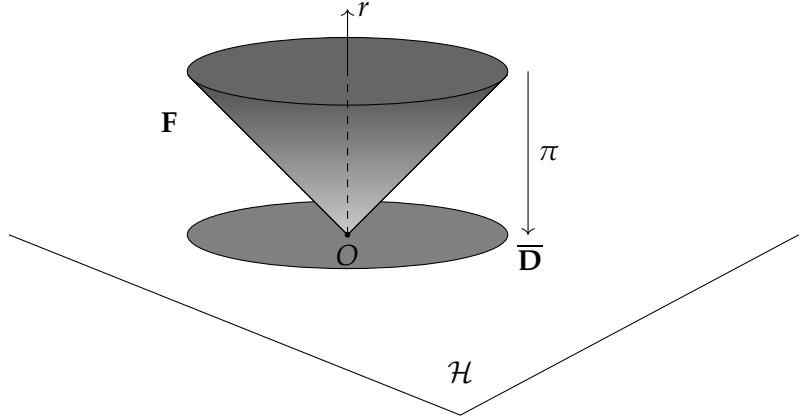

 Following notations in \cite{MR4117572}, we denote by $i\colon \D\to\D_1$ the map given by $i(x)(y)=d(x,y)-d(x,0)$ where $d$ is the hyperbolic metric on $\mathbf{D}$. This is $(\varphi_{y,0}(x))_{y\neq0}$ in the above notations i.e., this corresponds to the point $x_0=0$ in Remark \ref{X_1}. It is proved in \cite[Theorem 3.3]{MR4117572} that any function in $\overline{i(\D)}$ is given by the formula, for $y\in\D$:
 
 \begin{equation}\label{hori}\xi_{x,r}(y)=\log\left(\frac{1-\langle x,y\rangle+ \sqrt{\left(1-\langle x,y\rangle\right)^2-\left(1-||y||^2\right)\left(1-r^2\right)}}{(1+r)\sqrt{1-||y||^2}}\right)\end{equation}
 where $x,r$ are uniquely determined elements of $\overline{\D}\times[0,1]$ such that $||x||\leq r$. Actually,  $\varphi_{y,0}(x)$ coincides with $\xi_{x,r}(y)$ when $r=||x||$ and the Busemann function vanishing at 0 associated to $x\in \partial \D$ is $\xi_{x,1}$. The points that do not come from $\overline{\D}$ are thus the one corresponding to $x\in\D$ and $r>||x||$.

 \begin{prop}\label{horihyp}The horicompactification $\widehat{\H}$ is homeomorphic to the frustum $\F$ (and the projection $\pi\colon \F\simeq\widehat{\H}\to \overline{\D}$ is a continuous $G$-equivariant map).
 \end{prop}
 
 \begin{proof} It follows from Formula \eqref{hori} that the map $h\colon (x,r)\mapsto \left(\xi_{x,r}(y)-\xi_{x,r}(z)\right)_{y,z}$ from $\F$ to $\widehat{\H}$ is a continuous bijection between compact Hausdorff spaces and thus a homeomorphism. The continuity of the projection  $\pi\colon \widehat{\H}\to\overline{\D}$ is thus a direct consequence. The equivariance is a consequence of the construction of the map $f^{-1}\colon\H\to\D$ described in the paragraph about the Klein model. We detail the computation below.
 
 For $z\in\D\subset\mathcal{H}_-$, we denote by $\tilde{z}$ the point $e_0+z$. Let us denote by $\alpha(g)$ the action of $g\in\Isom(\H)$ on $\D$. More precisely, $\alpha(g)(x)=f^{-1}gf(x)$. Let us write 
 
$$ \xi_{x,r}(y)=\log\left(\frac{(\tilde x,\tilde y)+ \sqrt{(\tilde x,\tilde y)^2-Q(\tilde y)\left(1-r^2\right)}}{(1+r)\sqrt{Q(\tilde y)}}\right)$$

and observe that if we multiply $\tilde y$ by some positive constant, we get the same value. If we identify $g\in\Isom(\H)\simeq\PO(1,\infty)$ with an element of $\OO(1,\infty)$ which preserves the upper sheet of the hyperboloid, one has $(g\tilde y,e_0)\widetilde{\alpha(g)(y)}=g\tilde y$. So, if we set $\lambda=(g\tilde y,e_0)^{-1}$, $\mu=(g^{-1}\tilde x,e_0)>1$ and $\rho=\sqrt{1-\frac{1-r^2}{\mu^2}}$, we have

\begin{align*} \xi_{x,r}(\alpha(g)(y))&=\log\left(\frac{(\tilde x,\widetilde{\alpha(g)(y)})+ \sqrt{(\tilde x,\widetilde{\alpha(g)(y)})^2-Q\left(\widetilde{\alpha(g)(y)}\right)\left(1-r^2\right)}}{(1+r)\sqrt{Q\left(\widetilde{\alpha(g)(y)}\right)}}\right)\\
&=\log\left(\frac{(\tilde x,\lambda g\tilde y)+ \sqrt{(\tilde x,\lambda g\tilde y)^2-Q(\lambda g\tilde y)\left(1-r^2\right)}}{(1+r)\sqrt{Q(\lambda g\tilde y)}}\right)\\
&=\log\left(\frac{(\tilde x, g\tilde y)+ \sqrt{(\tilde x,g\tilde y)^2-Q( g\tilde y)\left(1-r^2\right)}}{(1+r)\sqrt{Q( g\tilde y)}}\right)\\
&=\log\left(\frac{(g^{-1}\tilde x,\tilde y)+ \sqrt{(g^{-1}\tilde x,\tilde y)^2-Q(\tilde y)\left(1-r^2\right)}}{(1+r)\sqrt{Q(\tilde y)}}\right)\\
&=\log\left(\frac{\left(\mu\widetilde{\alpha(g^{-1})(x)},\tilde y\right)+ \sqrt{\left(\mu\widetilde{\alpha(g^{-1})(x)},\tilde y\right)^2-Q(\tilde y)\left(1-r^2\right)}}{(1+r)\sqrt{Q(\tilde y)}}\right)\\
&=\log\left(\frac{\left(\widetilde{\alpha(g^{-1})(x)},\tilde y\right)+ \sqrt{\left(\widetilde{\alpha(g^{-1})(x)},\tilde y\right)^2-Q(\tilde y)\left(1-\rho^2\right)}}{(1+\rho)\sqrt{Q(\tilde y)}}\right)-\log\left(\frac{1+r}{\mu(1+\rho)}\right)\\
&=\xi_{\alpha(g^{-1})(x),\rho}(y)+\xi_{\alpha(g^{-1})(x),\rho}(0).\end{align*}

This computation shows that $g^{-1}\cdot h(x,r)=h(\alpha(g^{-1})(x),\rho)$ and we have the following commutative diagram giving the equivariance.
\begin{center}
\begin{tikzcd}
\widehat{\H} \arrow[r, "g^{-1}"] \arrow[d, "\pi\circ h^{-1}"]
&\widehat{\H} \arrow[d, "\pi\circ h^{-1}"] \\ 
\overline{\D} \arrow[r,"\alpha(g^{-1})" ]
&\overline \D \end{tikzcd}
\end{center}\end{proof}
 
 \begin{rem} This horicompactification is a $G$-flow but not a minimal one since the sheet $r=1$ is $G$-invariant and homeomorphic to $\overline{\D}$ via $\pi$.\end{rem}
 
  \subsection{Horicompactification of the Hilbert space} Let $\sigma_\mathcal{H}$ denote the map
  
  $$\begin{matrix}
  
 \mathcal{H}&\to&\D\\
  x&\mapsto&\frac{x}{\sqrt{1+||x||^2}}
  \end{matrix}$$
  
  which can be understood geometrically in the following way. Let $\mathcal{H}’$ be the Hilbert space $\mathcal{H}\oplus\R$. We identify $x\in\mathcal{H}$ with $(x,1)\in\mathcal{H}’$ and $y\in\D$ with $(y,0)\in\mathcal{H}’$. In this way, $\sigma_\mathcal{H}$ is the composition of the stereographic projection on the unit sphere in $\mathcal{H}’$ centered  at the origin and the vertical projection $\mathcal{H}’\to\mathcal{H}$.  
 
 \begin{figure}[h]
 \begin{center}
 \begin{tikzpicture}[scale=3]
  \coordinate (O) at (0,0);

    \def\rx{0.71}
    \def\ry{0.15}
    \def\z{0.725}
   \draw[fill = black!40] (O) ellipse ({\rx} and {\ry});

  \filldraw (O) circle (.01) node[below] {$O$};

  \draw[->] (0,\z)--(0,1) node[right] {$\mathbf{R}$} ;
  \draw[dashed] (O)--(0,\z);
  \draw (\rx,0) node[right] {$\overline{\mathbf{D}}$};
  
 \filldraw[ball color=black!40,opacity=.5] (-\rx,0) arc [start angle=-180, end angle = 0,
    x radius = {\rx}, y radius = {\ry}]--(\rx,0) arc (0:180:\rx)--cycle ;
     \draw (\rx,0) arc (0:180:\rx) ;


    \draw ( .5,-.7) node[above] {$\mathcal{H}$} ;
    \draw (-1.5,0) -- (.5,-.8)--(2,0);

    \draw (-1.5,\z) -- (.5,-.8+\z)--(2,\z);    
    
  \def\l{0.707}
  \draw  (-\rx,\rx)--(-\l*\rx,\l*\rx);
  \draw[dashed] (O)--(-\l*\rx,\l*\rx);
  \draw[dashed] (-\l*\rx,0)--(-\l*\rx,\l*\rx);
  \draw[fill]  (-\rx,\rx) circle (.01) node[above] {$(x,1)$};
   \draw[fill]  (-\l*\rx,0) circle (.01);
   \draw (-\l*\rx,-.1) node[below] {$\sigma_\mathcal{H}(x)$};
\end{tikzpicture}
\caption{The geometric interpretation of the map $\sigma_\mathcal{H}$.}
\end{center}
\end{figure}
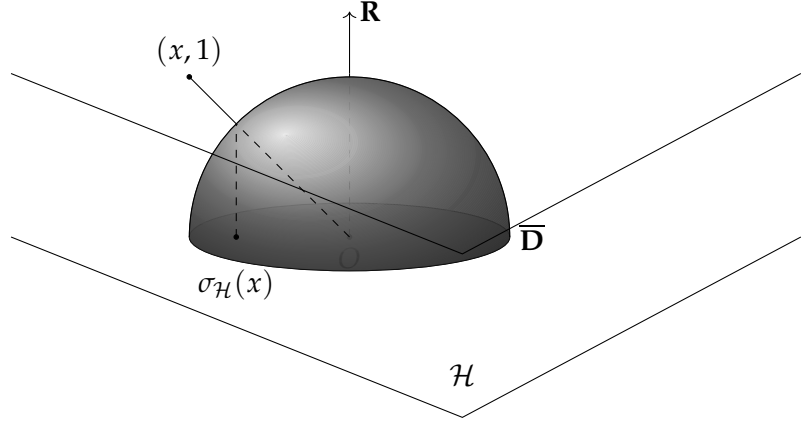

 \begin{prop}The horicompactification $\widehat{\mathcal{H}}$ is homeomorphic to the frustum $\F=\{(x,r)\in\overline{\D}\times[0,1], \ ||x||\leq r\}$.
 \end{prop}
 
 As before, we denote by $i\colon \mathcal{H}\to\mathcal{H}_1$ the map corresponding to $x_0=0$ in Remark \ref{X_1}. More precisely, $i(x)(z)=||x-z||-||x||$. Since the inverse map $\sigma_\mathcal{H}^{-1}\colon \D\to\mathcal{H}$ is given by $\sigma_\mathcal{H}^{-1}(y)=\frac{y}{\sqrt{1-||y||^2}}$, a computation shows that 
 $$i(\sigma_\mathcal{H}^{-1}(y))(z)=\sqrt{\frac{||y||^2}{1-||y||^2}-2\frac{\langle z,y\rangle}{\sqrt{1-||y||^2}}+||z||^2}-\frac{||y||}{\sqrt{1-||y||^2}}$$
 for $y\in\D$ and $z\in\mathcal{H}$. Following \cite{MR3998735} and considering a sequence $(y_n)$ such that $y_n$ weakly converges to some $y\in\overline\D$ and $||y_n||$ converges to $r\in[0,1]$ one gets that any element of $\widehat{\mathcal{H}}$ is given by
 
 $$\xi_{y,r}(z)=\sqrt{\frac{r^2}{1-r^2}-2\frac{\langle z,y\rangle}{\sqrt{1-r^2}}+||z||^2}-\frac{r}{\sqrt{1-r^2}}$$
 $||y||\leq r<1$ and this formula collapses to $$\xi_{y,1}(z)=-\langle y,z\rangle$$ when $r=1$.
 
 The formulas above show that the map $(x,r)\mapsto \xi_{x,r}$ is a continuous bijection between the compact Hausdorff spaces $\F$ and $\mathcal{H}$, and thus is a homeomorphsim. 

\begin{rem} For $x=0$ and $r=1$, one gets the zero function and in particular a global fixed point for the action of $\Isom(\mathcal{H})$ on $\widehat{\mathcal{H}}$.
 
One can try to understand how the action of $\Isom(\mathcal{H})$ on $\widehat{\mathcal{H}}$ translates to an action of $\Isom(\mathcal{H})$ on $\F\subset \overline{\D}\times [0,1]$. The action of the orthogonal group $O$ is simple to describe. For $g\in O$ and $(x,r)\in\F$, $g(x,r)=(g(x),r)$.
 
For the translation group $(\mathcal{H},+)$, the situation is different. Let $\tau_v$ be the translation with vector $v\in\mathcal{H}$ then $$\tau_v\cdot(x,r)=\left(\frac{1}{\sqrt{1+\lambda^2}}\left(\frac{x}{\sqrt{1-r^2}}+v\right),\frac{\lambda}{\sqrt{1+\lambda^2}}\right)$$ where $\lambda=\sqrt{\frac{r^2}{1-r^2}+\|v\|^2-2\langle v,\frac{x}{\sqrt{1-r^2}}\rangle}$. 
Whereas the hyperbolic and Hilbert horicompactifications are homeomorphic, one can observe that contrarily to the hyperbolic space, the projection $\F\to\overline{\D}$ does not induce an action of $\Isom(\mathcal{H})$ since the first component of $\tau\cdot(x,r)$ does depend on the second one.
 \end{rem}

\section{A Polish group}\label{secpolish}

Let $\tau$ be the topology of pointwise convergence on $\Isom(\H)$, that is the one induced by the pseudo-metrics $(g,h)\mapsto d(gx,hx)$ with $x\in \H$. This topology is Polish since $\H$ is separable. We call this topology "the Polish topology". The name will be justified later, since we will prove there is a unique Polish topology on $\Isom(\H)$.

Let us recall that the strong operator topology on $\GL(\mathcal{H})$ (the set of all bounded invertible operators on $\mathcal{H}$) is given by the family of pseudo-metrics $(g,h)\mapsto ||g(x)-h(x)||$ where $x$ varies in $\mathcal{H}$.

\begin{prop}\label{strong} The Polish topology coincides with the strong operator topology coming from the embedding $\Isom(\H)\leq \GL(\mathcal{H})$.
\end{prop}

\begin{proof} Let us embed $\H$ as $\{x\in \mathcal{H},\ Q(x)=1,\ x_0>0\}$. The hyperbolic metric and the Hilbert metric give rise to the same topology on $\H$. Thus, the Polish topology is weaker than the strong operator topology. The converse holds because, $\H$ is total in $\mathcal{H}$ and any converging sequence for the Polish topology is actually bounded for the operator norm. \end{proof}

%

It is proved in \cite[Theorem 3.14]{duchesne2019representations} that the group $\Isom(\H)$ is topologically simple but not abstractly simple.

Let $\H^n$ be the hyperbolic space of dimension $n>1$. A \emph{standard embedding} of $\Isom(\H^n)$ in $\Isom(\H)$  comes from a totally geodesic embedding $\varphi$ of $\H^n$ in $\H$. That is, the action of $\Isom(\H^n)$ on $\H$ is such that $\varphi$ is equivariant and the action is trivial on the orthogonal complement of the image of $\varphi$. 
\begin{lem} The group $\Isom(\H)$ is the completion of the union of all standard embeddings of $\Isom(\H^n)$.\end{lem}

 \begin{proof}It is shown in \cite[Proposition 3.10]{duchesne2019representations} that for any $g\in\Isom(\H)$ and points $x_1,\dots,x_n\in \H$, there is a standard embedding $\Isom(\H^k)\hookrightarrow\Isom(\H)$ and $h\in\Isom(\H^k)$ such $h(x_i)=g(x_i)$ for all $I$.
 \end{proof}

\begin{defn} A \emph{symmetry} is a non-trivial involutive isometry of $\H$.
\end{defn}

Let us observe that a symmetry $\sigma$ has non-empty fixed points set $F_\sigma$ (simply because orbits are bounded and $\H$ is \cat, see \cite[Corollary II.2.8]{MR1744486}). This set is a closed totally geodesic subspace (by uniqueness of the geodesic line through two distinct points, any geodesic line through two fixed points is pointwise fixed). The differential $d_x\sigma$ of  $\sigma$ at a fixed point $x\in F_\sigma$ is the orthogonal symmetry of $T_x\H$ (the tangent space at $x$) with $T_xF_\sigma$ as fixed points set. 

For any totally geodesic subspace $E\subset \H$ (maybe reduced to a point), we define the symmetry $\sigma_E$ whose fixed point set is exactly $E$. For a point $x\in \H$, there is a correspondence between orthogonal symmetries of the tangent space $T_x\H$ and symmetries of $\H$ fixing $x$. Any two symmetries are conjugate in $\Isom(\H)$ if and only if the $\pm1$-eigenspaces of their differentials at fixed points have the same dimensions. 

Cartan-Dieudonné theorem tells us that any element of $\Isom(\H^n)$ is the product of at most $n+1$ symmetries with respect to hyperplanes. If we allow more general symmetries and go to infinite dimensions, we get the following bounded generation result.

\begin{lem} An isometry of $\H$ is a product of at most 5 symmetries.\end{lem}

\begin{proof} It is well known that any element if the orthogonal group $O$ is a product of at most 4 symmetries. If $g\in\Isom(H)$, choose $x\in H$ 	and let $m$ be the mid-point of $[x,g(x)]$ and $\sigma_m$ be the symmetry at $m$. Then $\sigma_m\circ g$ fixes $x$ and the result follows.
\end{proof}


\subsection{Cartan decomposition}

Let $\mathfrak{o}(1,\infty)=\mathfrak{k}\oplus\mathfrak{p}$ be the Cartan decomposition of the Lie algebra $\mathfrak{o}(p,\infty)\leq\L(\mathcal{H})$ associated to the point $e_0\in\mathcal{H}$ and let us define $O=\Stab(e_0)$ and $P=\exp( \mathfrak{p})$.  Let us observe that $O$ is isomorphic to the orthogonal of the separable Hilbert space $\mathcal{H}_-$. Let $\varphi_{e_0}\colon T_{e_0}\H\to \mathfrak{p}$ be the identication between the tangent space and the subpaces $\mathfrak{p}$ of the Lie algebra. When $\mathfrak{p}$ is endowed with the Hilbert-Schmidt metric and $T_{e_0}\H$ is endowed with its Riemannian metric then $\varphi_{e_0}$ is a linear isometry (up to a scalar multiplication) between separable Hilbert spaces. With this identification, for any $v\in T_{e_0}\H$ one has $\exp(v)=\exp(\varphi_{e_0}(v))e_0$ where the first exponential is the Riemannian one and the second one is the exponential of bounded operators. We refer to \cite{MR3044451} for details.

Let us endow $O,P$ with the induced topology from the Polish topology on $G$. With this topology $\exp\colon \mathfrak{p}\to P$ is a homeomorphism. We endow the product $O\times P$ with the product topology.

\begin{prop}\label{Cartan} The following map  is a homeomorphism.

$$\begin{matrix}
O\times P&\to& G\\
(k,p)&\mapsto& pk
\end{matrix}$$
\end{prop}

\begin{proof}Since $\H$ is a simply connected manifold of non-positive curvature, the exponential map $\exp_{e_0}\colon T_{e_0}\H\to \H$ is a homeomorphism. Let $\log_{e_0}\colon \H\to\exp_{e_0}\H$ be the inverse of $\exp_{e_0}$. For $g\in G$, let $p=\exp(\varphi_{e_0}(\log_{e_0}(ge_0))$. By construction, $pe_0=ge_0$ and $p^{-1}g\in O$. So, the existence of the decomposition follows. Uniqueness follows from the fact that if $g=pk$ then $pe_0=ge_0$, and thus $p=\exp(\varphi_{e_0}(\log_{e_0}(ge_0))$.

Continuity is automatic since the Polish topology is a group topology. For the inverse, the map $g\mapsto p$ is continuous by composition of $\exp, \varphi$ and $\log_{e_0}$ and we conclude that $g\mapsto k$ is continuous because $k=p^{-1}g$.
\end{proof}

Since $O$ and $P$ are contractible, we obtain the following immediate consequence.
\begin{cor}The group $G$ is contractible.
\end{cor}

Let us now consider a similar decomposition for the isometries of the Hilbert space $\mathcal{H}$. As it is well known (it is a particular and easy case of the Mazur-Ulam theorem) and easy fact that any isometry of a Hilbert space is affine. It follows that $\Isom(\mathcal{H})$ splits as $O\ltimes \mathcal{H}$ where $\mathcal{H}$ is identified with the group of translations and $O$ is the orthogonal group identified with the stabilizer of the origin.

\begin{lem} The group isomorphism $\Isom(\mathcal{H})\simeq O\ltimes \mathcal{H}$ is, moreover, a homeomorphism between the pointwise convergence topology $\Isom(\mathcal{H})$ and the product topology of the strong operator topology on $O$ and the strong topology on $\mathcal{H}$.
\end{lem}

\begin{proof} The isomorphism is then $\varphi\colon(\rho,v)\in O\times\mathcal{H}\mapsto \tau_v\circ\rho$ where $\tau_v$ is the translation by $v\in\mathcal{H}$. The inverse  is given by $g\mapsto(\tau_{-g(0)}\circ g,\tau_{g(0)})$. The continuity of $\varphi$ and its inverse is an easy consequence of the joint continuity of the addition in $\mathcal{H}$.
\end{proof}

We also get the fact (as in finite dimension) that $\Isom(\mathcal{H})$ and $\Isom(\H)$ are homeomorphic but, of course, not isomorphic as groups.

\section{Automatic continuity}

Let $G$ be topological group. A subset $W\subset G$ is $\sigma$-syndetic if $G$ the union of countably many left translates of $W$. It is symmetric if $W=W^{-1}=\{w^{-1},\ w\in W\}$. The group $G$ has the \emph{Steinhaus property} if there is some natural integer $k$ such for any $\sigma$-syndetic symmetric subset $W\subset G$, $W^k$ contains an open neighborhood of the identity.

It is proved in \cite[Theorem 3]{MR3080189} that the orthogonal or unitary group of a real or complex Hilbert space of infinite countable dimension has the Steinhaus property (with $k=506$). It is a key fact for the following result.

\begin{thm} The Polish groups $\Isom(\H)$ and $\Isom(\mathcal{H})$ have the Steinhaus property.
\end{thm}

\begin{proof} Let $X=\H$ or $\mathcal{H}$ and $G=\Isom(X)$. For any $x\in X$, $\Stab(x)$ with the induced topology is isomorphic (as Polish group) to the orthogonal group $O$. Let $W$ be some symmetric $\sigma$-syndetic subset of $G$. By \cite[Lemma 4]{MR3080189}, $W^2\cap\Stab(p)$ is symmetric and $\sigma$-syndetic in $\Stab(p)$ for any $p \in X$. In particular, $W^{1012}$ contains an open neighborhood of the identity in $\Stab(p)$. 

Let us fix three distinct points $x,y,z\in X$ such that $z$ does lie on the geodesic line through $x$ and $y$. For a point $p\in X$, we say that some $g\in G$ is a rotation at $p$ if $g$ fixes $p$ and its differential (which coincides with the linear part of $g$ if $X=\mathcal{H}$) at $p$ has a codimension 2 subspace of invariant vectors and it acts as a standard rotation on the orthogonal plane.

Thus, one can find finitely points many $p_1,\dots,p_n$ and $\varepsilon>0$ such that $$\{g\in\Stab(p),\ d(gp_i,p_i)<\varepsilon,\ \forall i\leq n\}\subset W^{1012}$$ for $p=x,y$ and $z$. In particular, there is $\theta_0>0$ such that for any rotation $\rho$ at $p=x,y,z$ with angle $\theta<\theta_0$, $d(\rho p_i,p_i)<\varepsilon/3$ for any $i\leq  n$.

So, for $\alpha>0$ small enough and any $u,v\in B(x,\alpha)$ and at the same distance from $y$, there is $g\in W^{1012}\cap\Stab(y)$ such $g(u)=v$ and displacing the $p_i$’s by at most $\varepsilon/3$.

Since $z$ does not belong to the geodesic through $x$ and $y$, the set of distances $d(\rho(x),y)$ contains an interval $(d(x,y)-\lambda,d(x,y)+\lambda)$ with $\lambda\in(0,\alpha)$ where $\rho$ is a rotation centered at $z$ of angle $\theta<\theta_0$ in the totally geodesic plane containing $x,y,z$. 

Now, let $g\in G$ such that $d(gx,x)<\lambda$ and $d(gp_i,p_i)<\varepsilon/3$ for $i\leq n$. From above, one can find a rotation $\rho_1\in W^{1012}\cap\Stab(z)$ with angle less than $\theta_0$ such that $d(\rho_1(gx),y)=d(x,y)$. Moreover, one can find a rotation $\rho_2\in W^{1012}\cap\Stab(y)$ with angle less than $\theta_0$ such that $\rho_2(\rho_1(gx))=x$. Now, $\rho_2\rho_1g$ moves the $p_i$’s by at most $3\times \varepsilon/3$ and thus belongs to $\Stab(x)\cap W^{1012}$.

Finally, $g=\rho_1^{-1}\rho_2^{-1}(\rho_2\rho_1g)\in W^{3036}$ and the Steinhaus property is proved.
\end{proof}

As it is standard for Polish groups, the Steinhaus property implies the automatic continuity property, that is  any homomorphism to a separable Hausdorff group $\H$ is continuous \cite[Proposition 2]{MR2535429}. So, this proves Theorem~\ref{autcont} and implies that these groups have a unique Polish group topology.

\section{Amenable and extremely amenable subgroups}
The possibilities for amenable groups acting on $\H$ are well understood thanks to \cite[Theorem 1.6]{MR2558883}. For this theorem, the finiteness of the telescopic dimension is required and for $\H$ the telescopic dimension is exactly 1 since $\H$ is Gromov-hyperbolic. 
 
\begin{prop}\label{prop:amen} Let $G$ be an amenable topological group acting continuously by isometries on $\H$. Then $G$ has a fixed point in $\overline{\H}$ or stabilizes a geodesic line in $\H$. In particular, there is a subgroup of index at most 2, fixing a point in $\overline{\H}$.
\end{prop}

Let us first observe that the Polish group $\Isom(\H)$ is not amenable since there is no fixed point in $\H$ nor in $\partial \H$.  It does have a continuous isometric action on a Hilbert space without fixed point since the distance function on $\H$ induces a kernel of conditionally negative type \cite[\S7.4.2]{MR1852148}. Thus it does not have property FH and a fortiori it does not have property T.

For $x\in \H$, we denote by $\Stab(x)$ its stabilizer in $\Isom(\H)$. As observed previously, this group is isomorphic to the orthogonal group of a separable and since the latter is extremely amenable, we get immediately the following.

\begin{lem}\label{Lem:amen}For any $x\in H$, the Polish group $\Stab(x)$ is extremely amenable.
\end{lem}

Let $\xi\in \partial \H$, $x\mapsto \beta_\xi(x,x_0)$ the associated Busemann function vanishing at $x_0$ and $G_\xi$ its stabilizer in $\Isom(\H)$. The \emph{Busemann homomorphism} at $g\in G_\xi$, $\beta_\xi(g)$ is $\beta_\xi(gx_0,x_0)$ which does not depend on $x\in \H$.  This defines a continuous surjective homomorphism

$$\beta_\xi\colon G_\xi\to \RR.$$

Let $H_\xi\leq G_\xi$ be the kernel of the Busemann homomorphism, which is the set-wise stabilizer of horospheres around $\xi$. 

\begin{lem}\label{isomext} The  closed subgroup $H_\xi$ is isomorphic to $\Isom(\mathcal{H})$ as Polish group. In particular, it is extremely amenable.
\end{lem}

\begin{proof} In the model of the hyperbolic space $\H$ described in \cite[\S2.2]{MR3978389}, one sees that the hyperbolic metric on horospheres is a bijective function of an underlying Hilbert structure. So, any isometry preserving an horosphere  induces an isometry of the Hilbert structure. Conversely, any isometry of this Hilbert structure can be extended uniquely as an element of $H_\xi$. See \cite[\S2.4]{MR3978389} for the description of horospheres in this model.

This gives a group isomorphism $H_\xi\to\Isom(\mathcal{H})$. It is continuous since the topology on $\Isom(\mathcal{H})$ comes from the pointwise convergence for points in a fixed horosphere centered at $\xi$. One can prove easily by geometric means that the inverse is continuous as well but the automatic continuity of $\Isom(\mathcal{H})$ is a handy shortcut.
\end{proof}

\begin{lem} Let $\xi\in\partial \H$. The group $G_\xi$ is a closed amenable subgroup of $\Isom(\H)$\end{lem}

\begin{proof} Let us prove that $G_\xi$ is closed first. Let $x\in \H$ and $y\neq x$ such that $\xi$ is one extremity of the geodesic through these points. Let $(g_n)$ be a sequence converging to $g\in\Isom(\H)$. By definition of the topology $g_nx\to gx$ and $g_ny\to gy$. In particular, the geodesic $(g_nx,g_ny)$ converges uniformly on bounded subsets to $(gx,gy)$. So, if $g_n\in G_\xi$ then $g\in G_\xi$ as well because $\xi$ is an extremity of the geodesic line $(gx,gy)$. 

The group $G_\xi$ splits as a semi-direct product $H_\xi\rtimes \RR$. The elements of the group $\R$ can be realized as transvections along a fixed geodesic pointing to $\xi$.  Since $\RR$ is abelian, the amenability of $G_\xi$ follows from the one of $H_\xi$. \end{proof}

For a topological group $H$ we denote by $M(H)$ its universal minimal flow.

\begin{lem}\label{Lem:factor}Let $E$ be a topological extension of a quotient group $Q$ by an extremely amenable normal subgroup $A$. Then any minimal action of $E$ on a compact space factorizes to a minimal $Q$-flow. In particular  $M(E)\simeq M(Q)$. 
\end{lem}

\begin{proof} Let $X$ be some minimal $E$-flow. By extreme amenability of $A$, this subgroup has a fixed point $x$. Since $A$ is normal, any point $y$ in the same orbit as $x$ is a $A$-fixed point. Actually if $y=gx$ with $g\in E$, $ay=g(g^{-1}ag)x=gx=y$ for any $a\in A$. By minimality of the action, the orbit of $x$ is dense and thus $A$ acts trivially on $X$ and the action factorizes into a $Q$-action.
\end{proof}

In the particular case where $G_\xi$ is the  topological semi-direct product $H_\xi\rtimes\R$, we get the following identification of the universal minimal $M(G_\xi)$.

\begin{prop} The universal minimal flow $M(G_\xi)$ is homeomorphic to $M(\R)$.\end{prop}

\begin{rem}\label{minR}The universal minimal flow $M(\R)$ can be easily described from the Stone-\v{C}ech compactification of the integers $\beta\Z$. Actually, it is merely the suspension from $\Z$ to $\R$ of the extension to $\beta\Z$ of the shift map $n\mapsto n+1$. One may look at \cite{MR1357536} for details. Let us observe that this universal minimal space is not metrizable since $\beta\Z$ is not.
\end{rem}

A topological group $H$ is said to be \emph{strongly amenable} if any proximal minimal $H$-flow is trivial. For example, all abelian groups are strongly amenable \cite[\S II.4]{MR0474243}.

\begin{cor}\label{cor:strongamen} The group $G_\xi$ is strongly amenable.\end{cor}

\begin{proof} Let $X$ be some proximal $G_\xi$-flow. By Lemma~\ref{Lem:factor}, this is a minimal proximal $\R$-flow as well and thus it is reduced to a point.
\end{proof}

\section{Universal strongly proximal minimal flow }

Since $\mathcal{H}$ is separable, it is well known that the weak topology on the closed unit ball $\overline{\D}\simeq\overline{\H}$ is compact and metrizable. 

Let us recall that a flow  $X$ is \emph{strongly proximal} if the closure (for the weak-* topology on the space of probability measures $\Prob(X)$) of every orbit in $\Prob(X)$ contains a Dirac mass.

In the proof of the following proposition, we use angles $\angle_{p}(x,y)$ between points $x,y\in\overline{\H}$ at $p\in \H$. See \cite[I.2]{MR1744486} for a definition and basic properties of these angles. Let us observe that these angles coincide with the Riemannian angles of the tangent vectors of $[p,x]$ and $[p,y]$ in the tangent space $T_p\H$. One can also define them in the following way: Let $u,v$ be the initial vectors of the hyperbolic segments  $[p,x]$ and $[p,y]$ then $\cos(\angle_p(x,y))=-(u,v)$  \cite[I.2]{MR1744486}\footnote{They use the opposite of $Q$.}. 

\begin{lem}\label{continuity_angle} Let $x,y$ be distinct points of $\H$ and $(x_n),(y_n)$ be sequences in $\H$ converging strongly to $x$ and $y$. Let us fix $r>0$. Then for any $\varepsilon>0$, there is $N$ such that for all $n\geq N$  and $z\in\overline{\H}\setminus B(x,r)$, $$|\angle_{x_n}(y_n,z)-\angle_x(y,z)|<\varepsilon.$$\end{lem}

\begin{proof}Let $u_n$ and $v_n$ be the initial vectors at $x_n$ of the segments $[x_n,z]$ and $[x_n,y_n]$. So $\cos(\angle_{x_n}(z,y_n))=-(u_n,v_n)$. These initial vectors can be expressed as $u_n=\frac{z-(x_n,z)x_n}{(-Q(z-(x_n,z)x_n))^{1/2}}$ and $v_n=\frac{y_n-(x_n,y_n)x_n}{(-Q(y_n-(x_n,y_n)x_n))^{1/2}}$. Let us emphasize that $\|u_n\|$ is bounded by the supremum of the operator norms of the transvections $\tau_n$ from $e_0$ to $x_n$ because $\tau_n$ maps $e_0^{\bot}$ to $x_n^\bot$ and the preimage of $u_n$ by $\tau_n$ is a unit vector in $e_0^{\bot}$. The operator norm of $\tau_n$ is bounded above by $\cosh(d(e_0,x_n))$ and thus uniformly bounded. By homogeneity of the numerator and denominator, we may replace $z$ by the corresponding point $z’$ in $e_0+\overline{\D}$ (i.e., the point  $z’$ in $\mathcal{H}$ collinear to $z$ such that $(e_0,z’)=1$). 

Similarly, the initial vectors $u$ and $v$ of  the segments $[x,z]$ and $[x,y]$ can be expressed as $u=\frac{z’-(x,z’)x}{(-Q(z’-(x,z’)x))^{1/2}}$ and $v=\frac{y-(x,y)x}{(-Q(y-(x,y)x))^{1/2}}$. 

By uniform continuity of the arccosine function, it suffices to prove that for any $\alpha>0$, there is $N$ such that for any $n\geq N$, $|(u_n,v_n)-(u,v)|<\alpha$. Let us write

$$z’-(x,z’)x-(z’-(x_n,z’)x_n)=(x_n-x,z’)x_n+(x,z’)(x_n-x)$$

Now, since $(p,q)=\langle Jp,q\rangle\leq||p||\cdot||q||$ and $||z’||\leq 2$, 

\begin{equation}\label{ineqtruc}||z’-(x,z’)x-(z’-(x_n,z’)x_n)||\leq 2\left(||x_n||+||x||\right)||x_n-x||.\end{equation}

Let us observe that $|Q(z'-(x_n,z')x_n))|, |Q(z'-(x,z')x))|$ are bounded below by a constant depending only on $x$ and $r$ as soon $d(x_n,x)<r/2$. Actually, it suffices to consider the case $x=e_0$ and in this case $|Q(z'-(x,z')x))|\geq \tanh(r)^2$. Their are also bounded above independently of $z'$. 

So, together with Inequality \eqref{ineqtruc}, for any $\alpha_0>0$, one can find $N$ (independent of $z’$) such that $||u-u_n||<\alpha_0$ for any $n>N$. From the inequality 

\begin{align*}
|(u_n,v_n)-(u,v)|&\leq |(u_n,v_n)-(u_n,v)|+|(u_n,v)-(u,v)|\\
&\leq |(u_n,v_n-v)|+|(u_n-u,v)|\\
&\leq ||u_n||\cdot||v_n-v||+||u_n-u||\cdot ||v||,\\
\end{align*}
the strong convergence $v_n\to v$ and the boundedness of $\|u_n\|$, the desired domination follows.
\end{proof}


\begin{prop}\label{Gflow} The space $\overline{\H}$ is a metrizable strongly proximal minimal flow of $\Isom(\H)$ with 2 orbits and one of these orbits is comeager.
\end{prop}

\begin{proof} The main point is the continuity of the action. Since both $\Isom(\H)$ and $\overline{\H}$ are metrizable, it suffices to prove sequential continuity. Let $(x_n)$ be a sequence of points in $\overline{\H}$ converging to $x$ for the weak topology and $(g_n)$ be a sequence of isometries converging to $g\in \Isom(H)$ for the Polish topology. Since the topology on $\Isom(\H)$ is a group topology, it suffices to deal with the case $g=\Id$.

Let $U$ be some open half-space of $\overline{\H}$ containing $x$. Let $x_0$ be the projection of $x$ on the closed half-space $C=\overline{\H}\setminus U$. If $x\in\partial \H$ then $x_0$ is the minimum in $C$ of the Busemann function associated to $x$. Choose $W$ be some hyperplane orthogonal to the geodesic line $L$ through $x$ and $x_0$ that separates $x$ and $x_0$ and let $U'$ be the open half-space of $\overline{\H}$ associated to $W$ that contains $x$. Let $y= L\cap W$. 

By invariance of $U'$, $x_0$ and $y$ by rotation around the geodesic line through $x_0$ and $y$, there is $\alpha<\pi/2$  such that for all $z\in \overline{U'}$, $\angle_{x_0}(y,z)\leq \alpha$. Actually, this $\alpha$ can be obtained by a compactness argument in a hyperbolic plane containing $x_0,y$ and $z$. For $n$ large enough, $x_n\in U'$ and thus $\angle_{x_0}(x,x_n)=\angle_{x_0}(y,x_n)\leq \alpha$ for $n$ large enough. Since $g_nx_0\to x_0$ and $g_ny\to y$ (for the strong topology), by continuity of the angle (Lemma \ref{continuity_angle} with $r=d(x,y)$), for any $\varepsilon>0$, $n$ large enough and $z\in U'$,  
$$|\angle_{g_nx_0}(g_ny,z)-\angle_{x_0}(y,z)|<\varepsilon.$$

In particular for $z=g_nx_n$, $\angle_{x_0}(y,g_nx_n)\leq\angle_{g_nx_0}(g_ny,g_nx_n)+\varepsilon=\angle_{x_0}(y,x_n)+\varepsilon\leq\alpha+\varepsilon$. For $\varepsilon<1/2(\pi/2-\alpha)$, one gets that for $n$ large enough, $\angle_{x_0}(y,g_nx_n)\leq \pi/2-1/2(\pi/2-\alpha)$ and thus $g_nx_n\in U$. This proves the continuity of the action.


The two orbits are $\H$ and $\partial \H$ which are both dense. The minimality follows and it is proved in Lemma~\ref{lem_com} that $\partial \H$ is comeager.

Strong proximality follows from the fact that any proper closed subspace  is contained in some closed half-space and any closed half-space $X$ can be sent inside any open half-space via some hyperbolic element of $\Isom(\H)$.
\end{proof}

Let us recall that a closed subgroup $H$ of a topological group is \emph{coprecompact} if $G/H$ is precompact for the uniformity coming from the right uniform structure on $G$. This means that the completion $\widehat{G/H}$ is compact. This is equivalent to the fact that for any open neighborhood of the identity, $V$, there is a finite subset $F\subset G$ such that $VFH=G$.

In the remaining of this section, let us denote $G=\Isom(\H)$ and $G_\xi$ is the stabilizer of $\xi\in\partial\H$.
\begin{prop} The subgroup $G_\xi$ is coprecompact in $G$.
\end{prop}

\begin{proof} Let $V$ be some open neighborhood of the identity. By definition of the topology, it suffices to consider the case

$$V=\{g\in G,\ \forall i=1,\dots,n,\ d(gx_i,x_i)<\varepsilon\}$$

where $x_1,\dots,x_n\in \H$ and $\varepsilon>0$.

Let $\H_0$ be the minimal totally geodesic subspace containing the $x_i$’s and let $\H_1$ be some totally geodesic subspace containing $\H_0$ and $\xi$ with $\dim(\H_1)=\dim(\H_0)+1$. Let $c$ be the circumcenter of the $x_i$'s and let $\rho>0$ be the associated circumradius. By compactness  of the stabilizer of $c$ in $\Isom(\H_1)$ for the uniform convergence on compact subsets (i.e. the ,Lie group topology), one can find  finitely many elements of $\Stab(c)$ in $\Isom(\H_1)$ such that any other element of   $\Stab(c)$ lies at distance at most $\varepsilon$ for the metric $d(g,h)=\sup\{d(gx,hx),\ x\in B(c,\rho)\}$. Let us denote by $F$ the image of these elements under the standard embedding $\Isom(\H_1)\to G$. 

Let $g\in G$. The group $G_\xi$ acts transitively on pairs $(x,\H’)$ where $x$ is a point in the totally geodesic subspace $\H’$ of dimension $\dim(\H_0)+1$ such that $\xi\in\partial \H’$. Let $\H_1’$ be some totally geodesic subspace containing $g^{-1}(\H_0)$ and $\xi$ in its boundary. and with dimension $\dim(\H_0)+1$. So one can find $h\in G_\xi$ such that $h(g^{-1}(\H_0))\subset\H_1$  and $h(g^{-1}(c))=c$. Now, the restriction of $h\circ g^{-1}$ on $\H_0$ coincide with some element of $\Isom(\H_1)$ fixing $c$ and thus there is $f\in F$ such that the restrictions of $h\circ g^{-1}$ and $f^{-1}$ coincide up to $\varepsilon$ on $B(c,\rho)\cap \H_0$. In particular, if we set $v^{-1}=f\circ h\circ g^{-1}$ then $v^{-1}\in V$ and thus $v\in V$. So $g=vfh\in VFH$.
\end{proof}

\begin{thm}\label{strongprox} The universal strongly proximal minimal flow of $G$ is $\overline{\H}\simeq\widehat{G/G_\xi}$.
\end{thm}

\begin{proof} Let $X$ be some strongly proximal minimal flow. By amenability, $G_\xi$ fixes a point $x$. The orbit map $g\mapsto gx$ is uniformly continuous and thus induces a continuous $G$-map $\widehat{G/G_\xi}\to X$. It is surjective by minimality of $X$. This proves that $\widehat{G/G_\xi}$ is the strongly proximal minimal $G$-flow once we know that $\widehat{G/G_\xi}$ and $\overline{\H}$ are isomorphic as $G$-spaces.

Since $G$ acts continuously on $\overline{\H}$, we have a continuous map $\widehat{G/G_\xi}\to \overline{\H}$. Let us prove that the inverse of the bijection $G/G_\xi\to \partial\H$ is uniformly continuous by showing that the image of any Cauchy sequence  in $G/G_\xi$ is a Cauchy sequence in $\overline{\H}$ (this is equivalent since the spaces are precompact \cite{MR603371}).

Let $(g_n)$ be a sequence of elements in $G$ such that $g_n\xi$ converges in $\overline{\H}$ for the weak topology. We aim to show that $g_nG_\xi$ is Cauchy in $\widehat{G/G_\xi}$. It suffices to prove that for $x_1,\dots,x_k\in \H$, $\varepsilon>0$ and $V=\{v\in G,\ d(vx_i,x_i)<\varepsilon,\ \forall i=1,\dots,k\}$ there is $N\in\NN$ such that for $n,m\geq N$ there is $v_{n,m}\in V$ such that $g_mG_\xi=v_{n,m}g_nG_\xi$ i.e. $g_m\xi=v_{n,m}g_n\xi$. 

If $g_n\xi$ converges to some point in $\H$, we define $x_0$ to be this limit point. Otherwise, we define $x_0$ to be any point in $\H$. Since $G_\xi$ acts transitively on $\H$, we may and will assume that $g_n$ fixes $x_0$ for any $n\in\NN$. Let us define $v_{n,m}$ to be the rotation centered at $x_0$ such that $v_{n,m}g_n\xi=g_m\xi$. If $g_n\xi$ converges in $\partial \H$ then $g_n\xi$ converges in the cone topology and the angle of $v_{n,m}$ goes to 0 when $n,m\to\infty$. So $v_{n,m}$ converges uniformly to the identity on bounded subsets and $v_{n,m}\in V$ for $n,m$ large enough.

In the other case, the angle between the geodesic segment $[x_0,x_i]$ and the geodesic ray $[x_0,g_n\xi)$ goes to $\pi/2$ because $x_0$ is the limit point of $g_n\xi$. Thus, the angle between $[x_0,x_i]$ and the totally geodesic plane containing the geodesic rays $[x_0,g_n\xi)$ and $[x_0,g_m\xi)$ goes to $\pi/2$ for $n,m\to\infty$. In particular, $d(v_{n,m}x_i,x_i)\to0$ for $n,m\to\infty$. Thus $v_{n,m}\in V$ for $n,m$ large enough.
\end{proof}
Theorem~\ref{usp} is then a consequence of the homeomorphism $\overline{\D}\simeq\overline{\H}$ where the two spaces are endowed with the weak topologies.
 
\begin{thm}\label{prox} The universal proximal minimal flow of $G$ is $\overline{\H}\simeq \widehat{G/G_\xi}$.
\end{thm}

\begin{proof} The flow $\widehat{G/G_\xi}$ is a minimal proximal $G$-flow. The universal property follows from the fact that $G_\xi$ is strongly amenable (Corollary~\ref{cor:strongamen}).\end{proof}

\begin{rem} One can deduce Theorem~\ref{prox} from Theorem~\ref{strongprox} thanks to \cite[Theorem 1.7]{MR3509926} as well.
\end{rem}



\section{The universal minimal flow }\label{umf}
Let us denote by $M(G)$ the universal minimal Flow of the Polish group $G=\Isom(\H)$. The aim of this section is to describe this flow as the completion of some suspension. This suspension will be defined in two ways. The first definition will be more concrete but relies on some choices and a cocycle. The second one will be more pleasant but more abstract.

\begin{prop}\label{prop:decomp} There is a continuous $G$-equivariant map $\pi\colon M(G)\to\overline{\H}$ such that   for any $\xi\in\partial \H$, $M_\xi=\pi^{-1}(\{\xi\})$ is a minimal $G_\xi$-flow and thus a minimal $\R$-flow.
\end{prop}

\begin{proof} The existence of the map $\pi$ follows directly from the definition of the universal minimal flow.


For $\xi\in \partial \H$,   $M_\xi=\pi^{-1}({\xi})$ is a closed $G_\xi$-invariant subspace. Let $N$ be some closed minimal $G_\xi$-invariant subspace of $M_\xi$. Let $m\in N$ and $y\in M_\xi$. By minimality of the action of $G$ on $M(G)$, $y\in\overline{G\cdot m}$, so there is a net $(g_\alpha m)$ converging to $y$ with $g_\alpha \in G$. By continuity of $\pi$, $g_\alpha G_\xi$ converges to $G_\xi$ in $G/G_\xi$. So there is a net $(g'_\alpha)$ with $g'_\alpha\in G_\xi$ such that $g_\alpha g'_\alpha$ converges to the identity in $G$. By compactness of $N$, there is subnet  $(g'_\beta)$ such that $(g'_\beta)^{-1} m$ converges to $m'\in N$. Now $g_\beta m=(g_\beta g'_\beta)(g'_\beta)^{-1} m$ converges to $m'$ and thus $y=m'\in N$. So $M_\xi$ is a minimal $G_\xi$-flow.
\end{proof}

\begin{defn} Let $(X,\mathcal{U}_X)$ be a uniform space and $G$ a topological group with its right uniformity. An action by uniform isomorphisms is  \emph{bounded} if for any $U\in \mathcal{U}_X$ there is an open neighborhood $V\subseteq G$ such that for any $x\in X$, $V\cdot x\subset U(x)$  \end{defn}

Let us emphasize that a bounded action is continuous and any continuous action on a compact space is bounded  \cite[Remarks 2.17]{MR1900705}. We include a proof for completeness. One can also find this result in \cite[Proposition 3.3]{zbMATH03469702}). The same notion appeared under the name \emph{motion equicontinuity} in \cite{MR267038}.

\begin{lem} Let $X$ be a compact space with its unique compatible uniform structure. Any continuous action $G\times X\to X$ is bounded.
\end{lem}

\begin{proof} Let $U$ be some entourage in $X$. By continuity of the action, for any $x\in X$,  there is an identity neighborhood $W_x$ in $G$ and $V_x\subset U$ entourage such that for any $g\in W_x$ and $y\in V_x(x)$, $g(y)\in U(x)$. By compactness, there is $x_1,\dots,x_n$ such that $V_{x_1}(x_1)\cup\dots\cup V_{x_n}(x_n)=X$. Let us define $W=W_{x_1}\cap\dots\cap W_{x_n}$. For any $x\in X$, there is $x_i$ such that $(x,x_i)\in V_{x_i}$ and thus, for $g\in W$, $(gx,x_i)\in U$. So $(gx,x)\in U^2$.
\end{proof}

For a continuous action $G\action X$ on a uniform space, it is natural to ask when the action extends to the completion $\overline{X}$ of $X$. If $X$ is precompact, the boundedness of the action is required. Boundedness is actually a sufficient condition and such a result is well-know  by experts, see for example \cite[Lemma 4.5]{MR2476633}.

\begin{lem}\label{bounded} Let $X$ be a uniform space and let us assume that $G$ is a topological group with a bounded action on $X$ by uniform isomorphisms then this action extends to a continuous $G$-action on $\overline{X}$.
\end{lem}

\begin{proof} Each element of $G$ is a uniform isomorphism of $X$ and thus extends uniquely to a uniform isomorphism of $\overline{X}$. So we get an action of $G$ by uniform isomorphisms. By definition of the uniform structure on the completion, the extended action is bounded as well and thus continuous.
\end{proof}

Let $M$ be some minimal $\R$-flow with a free orbit (for example the universal minimal flow $M(\R)$). This orbit can be identified with the group $\R$ and the action $\R\action M$ extends the action by addition. So we denote this action additively. We aim to describe the universal minimal $G$-flow as some $G$-equivariant compactification of the homogeneous space $G/H_{\xi}$ where $\xi\in\partial \H$.

Let us denote by $\pi\colon G/H_\xi\to G/G_\xi$ the quotient map and $\beta\colon G/H_\xi\to \R\subset M$ be the map $gH_\xi\mapsto\beta_{g\xi}(gx_0,x_0)$, which is well defined because $H_\xi$ is exactly the stabilizer of the Busemann function associated to $\xi$. Actually, for $h\in H_\xi$, $\beta_{gh\xi}(ghx_0,x_0)=\beta_{gh\xi}(ghx_0,gx_0)+\beta_{gh\xi}(gx_0,x_0)=\beta_\xi(hx_0,x_0)+\beta_{g\xi}(gx_0,x_0)=\beta_{g\xi}(gx_0,x_0)$.\\

We denote by $\mathcal{U}$ the smallest uniform structure on $G/H_\xi$ making $\pi$ and $\beta$ uniformly continuous maps and let $\overline{G/H_\xi}$ be its completion with respect to $\mathcal{U}$. By definition, $\beta$ and $\pi$ extends to uniformly continuous maps on $\overline{G/H_\xi}$.

\begin{rem} The map $\beta\times \pi\colon G/H_\xi\to M\times G/G_\xi$ is injective and uniformly continuous with dense image. So, the completion $\overline{G/H_\xi}$ is isomorphic with $M\times \widehat{G/G_\xi}$ as uniform space ($ \widehat{G/G_\xi}$ is the completion of $G/G_\xi$ with respect to the right uniformity).
\end{rem}


\begin{prop} The uniform space $\overline{G/H_\xi}\simeq M\times \widehat{G/G_\xi}$ is a minimal $G$-flow.
\end{prop}

\begin{proof} It suffices to prove that $G$ acts continuously  and the action is minimal. Compactness is immediate.

Let us start by observing that $\pi$ is $G$-equivariant by definition and 
\begin{equation}\label{cocycle}
\beta(hgH_\xi)=\beta(gH_\xi)+c(h,g\xi)
\end{equation}

where $c$ is the cocycle $c\colon G\times \partial \H\to\R$, $c(h,\eta)=\beta_{h\eta}(hx_0,x_0)=\beta_\eta(x_0,h^{-1}(x_0))$.  The cocycle relation is $c(gh,\eta)=c(g,h\eta)+c(h,\eta)$. Let us observe that the restriction of $\beta$ on $G_\xi$ coincides with the Busemann homomorphism $\beta_\xi$ defined above. In particular $\beta(H_\xi)=0\in \R$.
The action on $M\times G/G_\xi$ is given by the following formula:
$$g(m,\eta)=(m+c(g,\eta),g\eta).$$

The Equation~\eqref{cocycle} shows that the map $\beta\times\pi$ is $G$-equivariant.

Let us show that the action by left multiplications $G\action G/H_{\xi}$ extends to a continuous $G$-action on is completion. Thanks to Lemma~\ref{bounded}, it suffices to prove that $G$ acts boundedly by uniform isomorphisms on $M\times G/G_\xi$. By compactness, we already know that the actions $\R\action M$ and $G\action \widehat{G/G_\xi}$ are bounded. In particular, for any entourage $V$ in $M$, there is $\varepsilon>0$ such that for any $m\in M$ and $r\in \R$, if $|r|<\varepsilon$ then $m+r\in V(m)$.


Since $c(g,\eta)=\beta_\eta(x_0,g^{-1}x_0)$, $|c(g,\eta)|\leq d(x_0,gx_0)$. Let us fix $(m,\eta)\in M\times G/G_\xi$.   For any $\varepsilon>0$, if $g$ is the neighborhood of the identity $\{g\in G,\ d(gx_0,x_0)<\varepsilon\}$, $|c(g,\eta)|<\varepsilon$ and thus the action $G\action M\times G/G_\xi$ is bounded.

Let us prove minimality. For $x,y\in\overline{G/H_\xi}$, we aim to show that $y\in\overline{Gx}$. First assume that $x,y\in\pi^{-1}(G/G_\xi)$. Since $G/G_\xi$ is a homogeneous $G$-space and $\pi$ is equivariant, we may assume that $\pi(x)=\pi(y)=G_\xi$. Let $(r_\alpha)$ be a net of real numbers such that $ \beta(x)+r_\alpha$ converges to $\beta(y)$ (such a net exists by minimality of the action $\R\action M$).  Let $(g_\alpha)$ be a net of transvections along a geodesic line with $\xi$ in its boundary at infinity such that $\beta_\xi(g_\alpha x_0,x_0)=r_\alpha$ for all $\alpha$. So $\beta(g_\alpha x)= \beta(x)+r_\alpha\to\beta(y)$ and thus $g_\alpha x\to y$. 

Now assume that $y\in\pi^{-1}\left(\widehat{G/G_\xi}\setminus G/G_\xi\right)$ and $x\in\pi^{-1}(G/G_\xi)$. By the above argument, it suffices to deal with the case where $\beta(x)=\beta(y)$. Let $\rho_n$ be a sequence of rotations centered at $x_0$ such that $\rho_n(\pi(x))\to\pi(y)$. Since for any $n$ and any $\eta$, $c(\rho_n,\eta)=0$, one has that $\beta(\rho_n x)=\beta(x)$ for any $n$ and thus $\rho_nx\to y$. Thus, we showed that a point $x\in \pi^{-1}(G/G_\xi)$ has a dense orbit.

So it remains to show that for some $x\in \pi^{-1}\left(\widehat{G/G_\xi}\setminus G/G_\xi\right)$, there is $y\in \overline{Gx}$ and $\pi(y)\in G/G_\xi$. It suffices to consider some sequence $g_n$ such that $g_n\pi(x)$ converges to some point in $G/G_\xi$ and extract a subnet $g_\alpha$ to guarantee that $\beta(g_\alpha x)$ converges as well. \end{proof}

\begin{rem} The cocycle $c\colon G\times G/G_\xi\to\R$ extends to a cocycle $\overline{c}\colon G\times \widehat{G/G_\xi}\to\R$ via the formula $\overline{c}(g,x)=\xi_{\hat x,1}(gx_0)-\xi_{\hat x,1}(x_0)$ where $\xi_{\hat x,1}$ is defined in Equation ~\eqref{hori} and $\hat x\in\D$ is the point corresponding to $x\in\overline\H\simeq\widehat{G/G_\xi}$. \end{rem}

Now, let us present the suspension as a quotient.  On the space $G\times M$, we consider the product uniform structure given by the right uniformity on $G$ and the unique uniform structure compatible with the topology on $M$. The group $G_\xi$ acts on this space by $h\cdot(g,m)=(gh^{-1},hm)$. We denote by $\sim$ the equivalence relation induced by this action. Let $R=\{(x,y),\ x\sim y\}\subset (G\times M)^2$. The equivalence relation is \emph{weakly compatible} with the uniform structure if for any entourage $D$, there is an entourage $D'$ such that $D'\circ R\circ D'\subseteq R\circ D\circ R$ (see \cite[Condition 2.13]{MR1069947}).

\begin{lem}\label{comp} The equivalence relation $\sim$ is weakly compatible with the uniform structure on $G\times M$.\end{lem}

Thanks to this weak compatibility, the quotient uniform structure on $\left(G\times M\right)/G_\xi$ is well defined. Moreover, if $Z$ is a any uniform space, a function $f\colon \left(G\times M\right)/G_\xi\to Z$ is uniformly continuous if and only if $f\circ p\colon G\times M\to Z$ is uniformly continuous, where $p\colon G\times M\to  \left(G\times M\right)/G_\xi$ is the projection map.
\begin{proof}[Proof of Lemma \ref{comp}] Let $U$ be some neighborhood of the identity in $G$ (thought as the entourage $\{(g,h),\ gh^{-1}\in U\}$ in $G$) and $V$ be some entourage of $M$. Let $D$  be the product entourage. Then $R\circ D\circ R$ is 

$$\left\{\left((g,m),(ug(hh')^{-1},n)\right),\ g\in G, h,h'\in G_\xi, u\in U, (h'm,h^{-1}n)\in V\right\}.$$

Similarly, for $U', V'$ and $D'$ the associated product entourage, $D'\circ R\circ D'$ is 

$$\bigcup_{n'\in M}\left\{\left((g,m),(u'ugh^{-1}, n)\right),\ g\in G,\ (m,n')\in V',\ (n',hn)\in V',\ u,u'\in U', h\in G_{\xi}\right\}.$$

So if one chooses $U'$ and $V'$ such that $(U')^2\subseteq U$ and $(V')^2\subseteq V$, one has $D'\circ R\circ D'\subseteq R\circ D\circ R$ (it suffices to take $h'=e$ in $R\circ D\circ R$ and replace $h$ by $h^{-1}$).
\end{proof}

The left multiplication on the first factor $G\action G\times M$ commutes with the action of $G_\xi$ and thus gives a continuous action by uniform isomorphisms on the quotient space $\left(G\times M\right)/G_\xi$. The quotient space $\left(G\times M\right)/G_\xi$ is the classical \emph{suspension} of the action $G_\xi\action M$. See \cite[I.1.3.j]{MR1928517} for generalities about suspensions. We call its completion, the \emph{completed suspension} and denote it by $S(M)$.

\begin{prop} The space  $S(M)$ is compact and isomorphic to $\widehat{G/G_\xi}\times M$ as $G$-flow. \end{prop}

\begin{proof} Let us show that the suspension is precompact. That is, for any entourage $D$, there are finitely many $x_i\in S(M)$ such that $S(M)=\cup_i D(x_i)$. Basic entourages of $S(M)$ are given by $p\times p(U\times V)$ where $U$ is an open neighborhood of the identity and $V$ is an entourage in $M$. Let us fix $U$ and $V$. By coprecompactness of $G_\xi$ in $G$, there are $g_1,\dots, g_n$ such that $\bigcup_{i=1}^n Ug_iG_\xi=G$ and by compactness of $M$, one can find $m_1,\dots,m_k$ such $M=\cup_{j=1}^kV(m_j)$.

We claim that $\cup_{i,j} p(Ug_i\times V(m_j))=S(M)$, which proves the precompactness of $S(M)$. Any element of $S(M)$ is some $p(g,m)$ for some $(g,m)\in G\times M$. There is $u\in U$, $i\leq n$ and $h\in G_\xi$ such that $g=ug_ih^{-1}$. Since $h\left(\cup_j V(m_j))\right)=M$, there is $j$ such that $m\in h(V(m_j))$ and thus $p(g,m)\in p(Ug_i\times V(m_j))$.

Let us prove the boundedness of the action $G\action \left(G\times M\right)/G_\xi$ and thus obtain an extended action $G\action S(M)$. Let us fix a basic entourage $D=p\times p(U\times V)$. For $u\in U$ and $p(g,m)\in (G\times M)/G_\xi$, $up(g,m)=p(ug, m) \in D(p(g,m))$ and thus the action is bounded. So $S(M)$ is a $G$-flow.

The map 

$$\begin{matrix}

\varphi\colon G\times M&\to&\widehat{G/G_\xi}\times M\\
(g,m)&\mapsto& \left(gG_\xi, m+c(g,\xi)\right)
\end{matrix}$$
is $G$-equivariant, invariant for the action of $G_\xi$ and uniformly continuous. So it induces a uniformly continuous $G$-equivariant map $\left(G\times M\right)/G_\xi\to \widehat{G/G_\xi}\times M$ (that we denote by $\varphi$ as well) and thus a $G$-equivariant uniformly continuous map $S(M)\to  \widehat{G/G_\xi}\times M$.  Its image is dense and compact, so it is surjective.

Let us prove that $\varphi$ has a uniformly continuous inverse on its image. On $G/G_\xi\times M$, the inverse map $\psi$ of $\varphi$ is given by $(gG_\xi, m)\mapsto p(g,m-c(g,\xi))$. Let  $U\times V$ be some basic entourage of $G\times M$. Let $U'$ be an open neighborhood of the identity such that $(U')^2\subseteq U$ and for any $u\in U'$, $\xi\in \partial \H$ and $m\in M$, $m+c(u,\xi)\in V(m)$ (such $U'$ exists by boundedness of the action $G_\xi\action M$). 

By precompactness of $G/G_\xi$, one can find $g_1,\dots,g_n\in G$ such that $\cup_{i=1}^n U''g_iG_\xi=G$ where $U''$ is some open neighborhood of $e$ with $(U'')^2\subset U'$. The maps $m\mapsto m-c(g_i,\xi)$ are uniformly continuous and thus one can find $V'\in \mathcal{U}_M$ such that $(m,n)\in V'$ implies $( m-c(g_i,\xi), n-c(g_i,\xi))\in V$ for all $i$. Now, let us consider $(gG_\xi,m),(hG_\xi,n)\in G/G_\xi\times M$ such that $g\in U''hG_\xi$  and $(m,n)\in V'$.  There is $g_i$ such that $g,h\in U'g_iG_\xi$. One has $\psi(gG_\xi,m)=p(ug_i,m-c(ug_i,\xi))$ and $\psi(hG_\xi,n)=p(vg_i,n-c(vg_i,\xi))$ for some $u,v\in U'$. In particular $ug_i(vg_i)^{-1}\in (U')^2$, $(m-c(ug_i,\xi),m-c(g_i,\xi))\in V$, $(m-c(g_i,\xi), n-c(g_i,\xi))\in V$ and $(n-c(g_i,\xi),n-c(vg_i,\xi))\in V$. So $\left(\psi(gG_\xi,m),\psi(hG_\xi,n)\right)\in p\times p(U\times V^3)$. In particular, $\psi$ is uniformly continuous and extends to a continuous inverse $\widehat{G/G_\xi}\times M\to S(M)$ and thus the two $G$-flows are isomorphic.\end{proof}

\begin{thm} The universal minimal $G$-flow is $S(M(\R))$.
\end{thm}

\begin{proof}Let $M(G)$ be the universal minimal $G$-flow. Let $M_\xi$ given by Proposition \ref{prop:decomp}. There is a $G_\xi$-equivariant continuous map $f\colon M(\R)\to M_\xi$. Let us define $\varphi\colon G\times M(\R)\to M(G)$ by $\varphi(g,m)=gf(m)$. For $h\in G_\xi$, $\varphi(gh^{-1},hm)=\varphi(g,m)$ and we claim that this map is uniformly continuous. So it defines a uniformly continuous map $\left(G\times M(\R)\right)/G_\xi\to M(G)$ and it extends to a map $S(M(\R))\to M(G)$. By construction, it is clearly $G$-equivariant. So, by  uniqueness of the universal minimal flow, this map is a homeomorphism   between $M(G)$ and $S(M(\R))$.

Let us prove the uniform continuity claim. Since $\widehat{G/G_\xi}\times M(\R)$ and $S(M(\R))$ are isomorphic, it suffices to prove that $\varphi$ is uniformly continuous on $G/G_\xi\times M(\R)$ seen as a subset of $S(M(\R))$. Actually, $\varphi(gG_\xi,m)=gf(m-c(g,\xi))$ for $(gG_\xi,m)\in G/G_\xi\times M(\R)$. By precompactness of  $G/G_\xi\times M(\R)$, it suffices to prove that the map is Cauchy continuous, that is it maps Cauchy filters to Cauchy filters or equivalently Cauchy nets to Cauchy nets \cite[Theorem 3]{MR603371}. So, let $(g_\alpha G_\xi,m_\alpha)_\alpha$ be some Cauchy net and $V$ be some symmetric entourage in $M(G)$. By boundedness of the action $G\action M(G)$, there is a symmetric open neighborhood of the identity $U$ such that for any $x\in M(G)$, $Ux\subseteq V(x)$. 

For each $\alpha$, let  us choose a representative $g_\alpha\in G$ of $g_\alpha G_\xi$ such that $g_\alpha\in G_{x_0}$, the stabilizer of some point $x_0\in \H$. This way, $c(g_\alpha,\xi)=0$ for all $\alpha$ and thus $\varphi(g_\alpha G_\xi,m_\alpha)=g_\alpha f(m_\alpha)$. Let us denote by $u_{\alpha,\beta}$ the rotation centered at $x_0$ such that $u_{\alpha,\beta}g_\beta\xi=g_\alpha\xi$. So $g_\alpha= u_{\alpha,\beta}g_\beta k_{\alpha,\beta}$ with $k_{\alpha,\beta}\in G_{x_0}\cap G_\xi\subset H_\xi$. Since $H_\xi$ fixes pointwise $M_\xi$, $\varphi(g_\alpha,m_\alpha)=u_{\alpha,\beta}g_\beta f(m_\alpha)$. 

Since the net is Cauchy, there is $\alpha_0$ such that for all $\alpha,\beta\geq\alpha_0$, $u_{\alpha,\beta}\in U$, thus, for $\alpha,\beta\geq \alpha_0$, $(g_\alpha f(m_\alpha), g_{\alpha_0}f(m_\alpha))\in V$ and  $(g_{\alpha_0}f(m_\beta), g_\beta f(m_\beta))\in V$.

Since $(m_\alpha)$ is Cauchy and $g_{\alpha_0}\circ f$ is uniformly continuous, $(g_{\alpha_0}f(m_\alpha))_\alpha$ is Cauchy as well and thus there is $\alpha_1\geq \alpha_0$ such that for $\alpha,\beta\geq \alpha_1$, $(g_{\alpha_0}f(m_\alpha),g_{\alpha_0}f(m_\beta))\in V$. Thus, for $\alpha,\beta\geq \alpha_1$, $(g_\alpha f(m_\alpha),g_\beta f(m_\beta))\in V^3$ and $\varphi$ is Cauchy continuous. 
%
%
%
\end{proof}

Since $M(\R)$ is not metrizable (Remark \ref{minR}), we deduce the following corollary. This non-metrizability can also be deduced from Lemma~\ref{not_min} below and results in \cite{MR4266370}.   

\begin{cor} The universal minimal space $M(G)$ is not metrizable.\end{cor}

We conclude this subsection by observing that $M(G)$ does not coincide with the Samuel compactification of some homogeneous space $G/H$ with $H$ extremely amenable subgroup. Thes maximal extremely amenable subgroups are the stabilizers $G_x$ of a point $x\in\H$ or  the horospherical groups $H_\xi$ for $\xi\in\partial \H$. So it suffices to prove the following.

\begin{lem}\label{not_min} The Samuel compactifications $\Sa(G/H)$ for a closed subgroup $H\leq G_{x}$ or $H\leq H_\xi$ are not minimal $G$-spaces.
\end{lem}

\begin{proof}We rely on \cite[Proposition 6.6]{MR4266370} where a characterization of minimality of the action of $G$ on $\Sa(G/H)$ for a closed subgroup $H\leq G$ is given. More precisely, the action is minimal if for any open neighborhood of the identity $U\subset G$, $UH$ is syndetic in $G$ i.e., finitely many left\footnote{We are left-handed where Zucker is right-handed.} translates of $UH$ cover $G$. So we prove that for some $U$, $UK$ and $UH_\xi$ are not syndetic since it implies that $UH$ is not syndetic. Let us take $U=\{g\in G,\  d(gx,x)<1\}$.  Let $F$ be any finite subset of $G$. 

Let us start with $G_x$. An element  in $FUG_x$ sends $x$ to a point at distance at most 1 from a point $fx$ with $f\in F$. Since $G$ acts transitively on $\H$ and $\H$ is unbounded, $FUG_x\neq G$. 

Let us continue with $H_\xi$, denote $R=\max_{f\in F}d(fx,x)$ and let  $g$ be the transvection along the geodesic line $L$ through $x$ to $\xi$ with translation length $\ell>0$. Assume that $G=FUH_\xi$. Thus, let us write $g=fuh$ and denote $l=fu$ with $f\in F$, $u\in U$ and $h\in H_\xi$. So, $l=gh^{-1}\in G_\xi$ and
\begin{align*}
d(l(x),x)&\leq d(l(x),f(x))+d(f(x),x)\\
&\leq d(u(x),x)+d(f(x),x)\\
&\leq 1+R
\end{align*}
Now, $h=l^{-1}g$ and thus $h(L)=l^{-1}(L)$. Let $\beta_\xi$ be the Busemann function with respect to $\xi$. Since $l^{-1}(x)$ and $h(x)$ are on the same geodesic line toward the point $\xi$, $d(l^{-1}(x),h(x))=|\beta_\xi(l^{-1}(x),h(x))|=|\beta_\xi(l^{-1}(x),x)|\leq d(l(x),x)$. So,
\begin{align*} 
d(x,h(x))&\leq d(x,l^{-1}(x))+d(l^{-1}(x),h(x))\\
&\leq 2 d(x,l(x))\\
&\leq 2(R+1)
\end{align*}
The following inequality gives a contradiction for $\ell>3(R+1)$.
\begin{align*}
d(g(x),x)&\leq d(lh(x),x)\\
&\leq d(lh(x),l(x))+d(l(x),x)\\
&\leq 2(R+1)+(R+1)= 3(R+1).
\end{align*}\end{proof}

\section{Minimality of the group topologies}
\subsection{Hyperbolic isometries}
In this subsection, we prove Theorem~\ref{topmin}, that is $\Isom(\H)$ with its Polish topology is minimal. Our proof is inspired by  the original proof of the minimality of the orthogonal group with its strong operator topology by Stojanov \cite{Stojanov_1984}. But the induction in Stojanov proof will be short-lived since we will only use the first step. Essentially we prove that the stabilizer $G_\xi$ of a point $\xi\in\partial \H$ in  $G=\Isom(\H)$ is closed and $G/G_\xi$ has a unique compact $G$-extension. 

We use the same terminology as in \cite{Stojanov_1984}. Let $X$ be a $G$-space. A $G$-\emph{compactification} $Y$ of $X$ is a $G$-flow with a continuous $G$-equivariant map $X\to Y$ with dense image. It is, moreover, a $G$-\emph{extension} if it is an homeomorphism on its image.

The stabilizer $G_\xi$ can be identified with the group of Möbius transformations of the Hilbert space $\mathcal{H}$ and thus splits as the semi-direct product $(\R\times O)\ltimes \mathcal{H}$ where $(\mathcal{H},+)$ is the group of translations, $O$ is the orthogonal group and $\R$ corresponds to positive homotheties.  See \cite[\S2]{MR3978389}.

%
%
\begin{prop}\label{Gext} The only non-trivial $G$-compactification of the sphere $\partial \H\simeq G/G_\xi$ is $\overline{\H}$ with its weak topology.\end{prop}

\begin{proof} Let $\partial\H\to C$ be such a $G$-compactification where $C$ is not reduced to a point. The action $G\action C$ is bounded by compactness of $C$. Thus the map $\partial \H\simeq G/G_\xi\to C$ is uniformly continuous, with respect to the right uniformity on $G/G_\xi$, and extends to a surjective $G$-map $\overline{\H}\simeq \widehat{\partial \H}\to C$. It suffices to prove that this map is injective to get that $C\simeq \overline{\H}$. It is injective on $\partial \H$ because otherwise the double transitivity of $G\action \partial \H$ (see for example \cite[Proposition 2.5.9]{MR3558533}) would imply that the image of $\partial \H$ collapses to a point. Since $G$ acts transitively on pairs $(x,\xi)\in \H\times\partial\H$, the same argument shows that $x$ and $\xi$ cannot have the same image.

Assume that two distinct points $x,y$ of $\H$ are sent to the same point. We claim that the image of $\H$ is reduced to this point. Actually, the stabilizer of a point acts transitively on each sphere around this point. So the whole sphere of radius $d(x,y)$ around $x$ is mapped to a unique point. For any $r\leq 2d(x,y)$, this sphere contains a point at distance $r$ from $y$. So, by the same argument, all the points in the closed ball of radius $2d(x,y)$ is mapped to a unique point. A proof by induction shows that for any $n\in \NN$, the closed ball of radius $n$ around $x_n$ is mapped to  the same point in $C$ where $x_n$ is $x$ if $n$ is even and $y$ otherwise. Since $\H$ is dense in $\overline{\H}$, this implies that the image is a point.
\end{proof}

\begin{rem} This proposition proves that the sphere $\partial \H$ is a \emph{weakly minimal} $G$-space as defined in \cite{Stojanov_1984}. This means that any $G$-compactifcation is a $G$-extension. Proposition~\ref{Gext} can be put in parallel with the weak minimality of the Hilbert sphere and the projective space as $O$-spaces where $O$ is the orthogonal group of $\mathcal{H}$ \cite[Proposition 4.2]{Stojanov_1984}.
\end{rem}

\begin{lem}\label{basis} The subsets $VP=\{vp,\ v\in V,\ p\in P\}$ where $P$ is the stabilizer of a point in $\partial \H$ and $V$ is an open neighborhood of the identity, is a sub-basis of identity neighborhoods in $G$.
\end{lem}

\begin{proof}It suffices to prove that for $x\in \H$ and $\varepsilon>0$, there is an identity neighborhood $V$ and point stabilizers $P_1,P_2,P_3$ such that for all $g\in VP_1\cap VP_2\cap VP_3$, $d(gx,x)<\varepsilon$. 

Let $X$ be a two-dimensional totally geodesic submanifold of $\H$ such that $x\in X$. Let $\xi_1,\xi_2\in \partial X$ such that $x$ lies on the geodesic line between these two points at infinity. Let $\xi_3 \in \partial X$ distinct from $\xi_1$ and $\xi_2$. Let $x_1,x_2$ be distinct points on the geodesic $(\xi_1,\xi_2)$ and $x_3$ be a point on $(x,\xi_3)$. Let $\alpha>0$ and $V_\alpha=\{g\in G,\ d(gx_i,x_i)<\alpha,\ \forall i=1,2,3\}$.  For any neighborhood $W$ of $\xi_i$, there is $\alpha>0$ such that for any $g\in V_\alpha$, $g\xi_i\in W$ and thus for any $g\in V_\alpha P_i$, $g\xi_i\in W$. 

We claim that for any $\varepsilon>0$ one can find $\alpha$ small enough such that $d(gx,x)<\varepsilon$ for any $g\in\cap_i V_\alpha P_i$. Assume this is not the case then we can find a sequence $(g_n)$ such that $d(g_nx,x)>\varepsilon$ for all $n\in\NN$ and for all $\alpha>0$, $g_n\in \cap_i V_\alpha P_i$ eventually. Since the pointwise stabilizer of the totally geodesic hyperbolic plane containing $\xi_1,\xi_2,\xi_3$ acts transitively on totally geodesic subspaces of dimension 5 containing $\xi_1,\xi_2,\xi_3$, we may assume that the $g_n\xi_i’$s belong to the boundary of a fixed totally geodesic subspace $Y$ of dimension 5 and actually that $g_n$ lies in the image of $\Isom(Y)$ in $G$ for all $n\in\NN$. Since the action of $\Isom(Y)$ on triples of distinct points of $\partial Y$ is proper (this is one of the first examples of convergence groups), the sequence $(g_n)$ is bounded in $\Isom(Y)$ and by local compactness of $\Isom(Y)$, one can extract a converging subsequence with some limit $g\in \Isom(Y)$. Since $g$ fixes each $\xi_i$, the restriction of $g$ on $X$  is trivial and thus $gx=x$. So we have a contradiction.
\end{proof}

Let us recall a standard fact about topological groups.
\begin{lem} Let $(G,\sigma)$ be a topological group and $H\leq G$ be a closed subgroup then the normalizer of $H$ in $G$ is closed.\end{lem}

\begin{proof} For $h\in H$, let $\varphi_h(g)=ghg^{-1}$. The map $\varphi_h\colon G\to G$ is continuous and the normalizer of $H$ is $\cap_{h\in H}\varphi_h^{-1}(H)$ and thus closed.\end{proof}

 Let $\sigma$ be any Hausdorff group topology on $G=\Isom(\H)$ and $\xi\in\partial\H$. 

\begin{lem}The subgroup of translations $T_\xi\simeq\mathcal{H}$ in $G_\xi$ is a closed  subgroup of $(G,\sigma)$.\end{lem}

\begin{proof} Let $t$ be some non-trivial translation of $G_\xi$. If $g\in G$ commutes with $t$, then it fixes the unique fixed point in $\partial\H$ of $t$, which is $\xi$. So its centralizer is a closed subgroup of $G$ included in $G_\xi$. From the semi-direct structure of $G_\xi$, it easy to see that the intersection of the centralizers of all $t\in T_\xi$ is exactly $T_\xi$ and thus this group is closed in $G$.
\end{proof}

\begin{prop}\label{cloclo} The subgroup $G_\xi$ is  a closed subgroup of $(G,\sigma)$.
\end{prop}

\begin{proof} The group $G_\xi$ is the normalizer of $T_\xi$. Actually, $T_\xi$ is normal in $G_\xi$ and conversely any element normalizing $T_\xi$ fixes the unique point at infinity fixed by all elements in $T_\xi$.\end{proof}

We can now prove that the Polish topology on $G$ is minimal.
\begin{proof}[Proof of Theorem ~\ref{topmin}] Assume that $\sigma$ is a  Hausdorff group topology coarser than the Polish topology $\tau$ on $G$. Let $P$ be the stabilizer of a point in $\partial \H$ which is a closed subgroup of $G$ for $\sigma$ by Proposition~\ref{cloclo}. 

Let us endow $G/P$ with the quotient topology $\sigma_P$ obtained from $\sigma$. Let $C$ be some compact $G$-extension of $G/P$ (which exists by \cite[Lemma 4.4]{Stojanov_1984}). By Proposition~\ref{Gext}, $C=\overline{\H}$ and thus the quotient topologies $\sigma_P$ and $\tau_P$ coincide on $G/P$. So, for any identity neighborhood $V$ for $\tau$, $VP$ is an identity neighborhood for $\sigma$. By Lemma~\ref{basis}, this implies that $\sigma$ is finer than $\tau$ and thus $\sigma=\tau$.
\end{proof}

\subsection{Euclidean isometries} We now prove that the Polish group $\Isom(\mathcal{H})$ is minimal. In finite dimension, the same result is obtained in \cite{MR1428116} with very different methods since the group is locally compact. We use the semi-direct decomposition $\Isom(\mathcal{H})=O\ltimes \mathcal{H}$ where $O$ is identified to the stabilizer of some origin in the Hilbert space $\mathcal{H}$ and $\mathcal{H}$ is identified to the subgroup of translations. In this identification, the action of $O$ on the Hilbert space corresponds to the action by conjugations on the subgroup of translations.

The abelian group $(\mathcal{H},+)$ (whith he norm topology) is not minimal. For example, the weak topology is a coarser group topology. To get the minimality of $G=\Isom(\mathcal{H})$, the action of the rotations in $O$ on $\mathcal{H}$ plays a key role and we use ideas from \cite{MR551694} where it is proved that the affine group of the real line $\R$ is minimal whereas the group $(\R,+)$ with its usual topology is not minimal.

\begin{proof}[Proof of theorem~\ref{topmineuc}] Let $\tau$ be the Polish topology on $G=\Isom(\mathcal{H})$ and $\sigma$ be a coarser Hausdorff group topology. We first observe that the stabilizer of any point is isomorphic to $O$ and by minimality of the strong topology on $O$, $\sigma|_O=\tau|_O$.

Let $u$ be a unit vector in $\mathcal{H}$ and $\sigma_u\in O$ be the associated symmetry fixing pointwise the orthogonal of $u$. A simple computation shows that for any $t\in\mathcal{H}$, the commutator $[t,\sigma_u]$ is $2\langle t,u\rangle u$. Let us denote by $\R u$ the line spanned by $u$ in $\mathcal{H}$. It is a closed subgroup of $(G,\sigma)$ since it is the commutator of the semi-direct product of the stabilizer of $u$ in $O$ with $\R u$. 

The map 
$$\begin{matrix}
\pi_u\colon \mathcal{H}&\to&\R u\\
t&\mapsto&[t,\sigma_u]
\end{matrix}$$
is a continuous map with respect to $\sigma$ on the domain and the target. Let $v$ be a unit vector orthogonal to $u$ and let $\rho_\theta$ be the rotation of angle $\theta$ in the base $(u,v)$. The map 
$$\begin{matrix}
\psi\colon\R\times \mathcal{H}&\to&\R u\\
(\theta,t)&\mapsto& \pi_u(\rho_\theta \pi_u(t)\rho_\theta^{-1})\end{matrix}$$

is continuous and $ \pi_u(\rho_\theta \pi_u(t)\rho_\theta^{-1})=4\cos(\theta)\langle u,t\rangle u$.

Let $U$ be an identity neighborhood of $(\R u,\sigma)$ such that $U-U\neq \R u$ (which exists since $\sigma$ is Hausdorff). By continuity of $\psi$, there is $\theta_0>0$ (smaller than $\pi/2$) and $V$ an identity neighborhood of $\mathcal{H}$ such that for any $(\theta, t)\in ]-\theta_0,\theta_0[\times V$, $\psi(\theta,t)\in U$. Let $\varepsilon=\min\{\cos(\theta_0/2)-\cos(\theta_0), 1-\cos(\theta_0/2)\}$. Since $\psi(\theta,t)-\psi(\theta_0/2,t)\in U-U$ for all $(\theta,t)\in ]-\theta_0,\theta_0[\times V$, for all $t\in V$, $]-4\varepsilon\langle u,t\rangle u, 4\varepsilon\langle u,t\rangle u[\in U-U$. Since $U-U\neq G$, $\{\langle u,t\rangle, t\in V\}$ is bounded. So the restriction of $\sigma$ on $\R u$ has bounded identity neighborhoods and since the restriction of $\tau$ on $\R u$ is locally compact, $\sigma$ and $\tau$ coincide on $\R u$. Using this fact for all unit vectors $u$, we get that $\sigma$ is finer than the weak topology on $\mathcal{H}$.

We claim there is a bounded neighborhood of the origin in $\mathcal{H}$ for $\sigma$. Otherwise, one can find a net $(t_\alpha)$  of $\mathcal{H}$ converging to $0$ for $\sigma$ and such that $||t_\alpha||\to\infty$. We observe that $t_\alpha$ does not lie in $\R u$ for $\alpha$ larger than some fixed $\alpha_0$ (otherwise, it would contradict the weak convergence). 

Let $\rho_\alpha$ be the rotation in the plane spanned by $u$ and $t_\alpha$ with angle $\theta_\alpha$ such that $\langle\rho_\alpha(t_\alpha),u\rangle\to+\infty$ and $\theta_\alpha\to0$. It is possible to find such an angle $\theta_\alpha$ because $$\langle\rho_\alpha(t_\alpha),u\rangle=\sin(\theta_\alpha)\langle t_\alpha,u_\alpha\rangle+\cos(\theta_\alpha)\langle t_\alpha,u\rangle$$ where $u_\alpha$ is the unit vector orthogonal to $u$ in the spanned of $u$ and $t_\alpha$ such that $\langle u_\alpha,t_\alpha\rangle>0$.  By weak convergence, $\langle t_\alpha,u\rangle\to0$ and thus $\langle t_\alpha,u_\alpha\rangle\sim||t_\alpha||\to +\infty$. So one can choose, for example, $\teta_\alpha=\arcsin\left(\log(||t_\alpha||)/\langle t_\alpha,u_\alpha\rangle\right)$. Since $\theta_\alpha\to0$, $\rho_\alpha$ converges to the identity for the strong topology. So $\rho_\alpha t_\alpha \rho_\alpha^{-1}$ converges to the identity and we have a contradiction with $\langle\rho_\alpha(t_\alpha),u\rangle\to\infty$.

So there is $R>0$ such that the ball $B(0,R)$ is a neighborhood of 0 in $\mathcal{H}$ for $\sigma$. Since this is a group topology, there is an open set $U$ containing 0 such that $U+U\subset B(0,R)$. In particular $U\subset B(0,R/2)$. Repeating this argument, we see that the collection of open balls around the origin is a collection of origin neighborhoods for $\sigma$. So $\sigma$ and $\tau$ coincides on $\mathcal{H}$.  

Since $\sigma$ and $\tau$ coincide on $O\simeq G/\mathcal{H}$ and on $\mathcal{H}$, they coincide on $G$ thanks to  \cite[Lemma 1]{MR551694}.
\end{proof}

\begin{rem}This yields another proof of Theorem~\ref{topmin}. One can prove first that the Polish topology on the group of Möbius transformations on $\mathcal{H}$, i.e., the group of similiarities $\left(\R\times O\right)\ltimes \mathcal{H}$ is minimal (which follows easily from Theorem~\ref{topmineuc}) and combine this fact with Proposition~\ref{Gext} and \cite[Lemma 1]{MR551694} to get the minimality of $\Isom(\H)$. 

\end{rem}

\section{Existence and lack of dense conjugacy classes}

A simple idea to separate conjugacy classes is to find a continuous non-consant conjugacy invariant. For finite dimensional linear groups, the spectrum is such an invariant. In our geometric setting, a natural invariant is the translation length. For a metric space $X$ and $g\in\Isom(X)$, the translation length is 

$$\ell(g)=\inf_{x\in X}d(gx,x).$$

\begin{lem} The translation length is upper semi-continuous on $\Isom(X)$ for the pointwise convergence topology.
\end{lem}

\begin{proof} This follows from the general fact that an infimum of continuous functions is upper semi-continuous.
\end{proof}

\begin{cor}\label{neutral} For any separable metric space $X$, if $g\in\Isom(X)$ has a dense conjugacy class then $\ell(g)=0$.\end{cor}

\begin{proof} Let $g_n$ be a sequence in the conjugacy class of $g$ converging to the identity. Since $\ell(g)=\limsup_{n\to\infty} \ell(g_n)\leq\ell(\Id)=0$, we have the result.
\end{proof}

%
The following theorem is surely well known but we provide a proof since we use some elements constructed in the proof.

\begin{prop}\label{densecc} The orthogonal and the unitary groups of a separable Hilbert $\mathcal{H}$ space have a dense conjugacy class.\end{prop}

\begin{lem}\label{finite rank} Let $\mathcal{H}$ be a real or complex Hilbert and $U(\mathcal{H})$ its orthogonal or unitary group. If $A\in U(\mathcal{H})$ and $x_1,\dots,x_k\in\mathcal{H}$ then there is an operator $A'\in U(\mathcal{H})$ which coincides with $A$ on each $x_i$ and which is the identity on a subspace of finite codimension.
\end{lem}

\begin{proof}Without loss of generality, we may assume that $(x_1,\dots, x_k)$ is free. Let $e_1,...,e_k$ be the basis obtained by the Gram–Schmidt process and let $e'_1,\dots,e'_k$ its image by $A$. We define $\mathcal{F}$ to be the (finite dimensional) span of the union of these two families. By completing these two orthonormal families, one can find an orthogonal or unitary operator $U_1$ of $\mathcal{F}$ mapping $e_i$ to $e'_i$ and thus $x_i$ to $U_0(x_i)$. Extending this operator by the identity on $\mathcal{F}^\bot$, we get $A'$.
\end{proof}

\begin{proof}[Proof of Proposition~\ref{densecc}] We prove the theorem in the complex case. In the real case, the proof is the same, using rotations instead of complex  homotheties with unitary ratio.

Let $(\lambda_n)$ be  a sequence of complex numbers in the unit circle that is dense (in the real case, we choose rotations with angles the arguments of the $\lambda_n$). Let us write $\mathcal{H}$ as an infinite orthogonal sum $\oplus_n\mathcal{H}_n$ where $\mathcal{H}_n$ is a closed subspace of infinite dimension. Let us define a unitary operator $U$ that is the multiplication by $\lambda_n$ on each $\mathcal{H}_n$. We claim that the conjugacy class of $U$ is dense in $U(\mathcal{H})$. That is, for any $U_0\in U(\mathcal{H})$, $x_1,\dots,x_k$ in the unit sphere of $\mathcal{H}$ and $\varepsilon>0$, there is $U'$ in the conjugacy class of $U$ such that for any $i=1,\dots,k$, 

\begin{equation}\label{fr}\|U'(x_i)-U_0(x_i)\|<\varepsilon.\end{equation}

Let us apply Lemma~\ref{finite rank} to $U_0$ and $x_1,\dots,x_k$. We get an operator $U'_0$ that coincides with $U_0$ on the span of the $x_i$'s and that is trivial on a finite codimension subspace. Let us denote by $U_1$ the restriction of $U'_0$ on $\mathcal{F}$, the orthogonal of this finite codimension subspace.

There is an orthonormal basis $f_1,\dots, f_l$ of $\mathcal{F}$ such that $U_1$ acts diagonally in this basis, multiplying each $f_j$ by some $\alpha_j\in S^1$. For each $j$, choose $\lambda_{i(j)}$ such that $|\alpha_j-\lambda_{i(j)}|<\varepsilon/k$. 
Now, find a unitary operator $T$ mapping some unit vector of $\mathcal{H}_{i(j)}$ to $f_j$ for each $j$ and set $U'$ to be $TUT^{-1}$. This way, each $f_j$ is an eigenvector of $U'$ with eigenvalue $\lambda_{i(j)}$ and thus for each $x$ in the unit sphere of $\mathcal{F}$, $\|U'(x)-U_1(x)\|<\varepsilon$. Restricting this equality to the $x_i$'s,  we get Inequality \eqref{fr}.\end{proof}

\begin{thm}\label{densecca} The Polish group $\Isom(\mathcal{H})$ has a dense conjugacy class.\end{thm}
\begin{proof} We prove that the element $U$ constructed in the proof of Theorem~\ref{densecc} has a dense conjugacy class in $\Isom(\mathcal{H})$. Thanks to Lemma~\ref{finite rank}  and the fact that translations act transitively, it suffices to approximate elements $g$ such that $g$ preserves a finite dimensional linear subspace and acts trivially on its orthogonal. Let us recall that all isometries of finite dimensional Hilbert spaces are semisimple, that is the infimum in the definition of the translation length is actually a minimum.

If $g$ has a fixed point, it is conjugated to an element of $O$, the stabilizer of the origin in $\mathcal{H}$. In this case, it lies in the closure of the conjugacy class of $U$ by Theorem~\ref{densecc}.

Now assume that $\ell(g)>0$. Let us write $g(x)=Ax+b$ where $A\in O$ and $b\in\mathcal{H}$. The Hilbert space $\mathcal{H}$ splits orthogonally as $\im(I-A)\oplus\ker(I-A)$. Let us set $A_0$ to be the restriction of $A$ to $\im(I-A)$ and let $b=b_0+b_1$ be the decomposition of $b$ with respect to this splitting. The isometry $g$ acts diagonally with respect to this splitting as $g_0\times g_1$ where $g_0(x_0)=A_0x_0+b_0$ and $g_1(x_1)=x_1+b_1$. Since $b_0\in\im(I-A_0)$, $g_0$ has a fixed point. Up to conjugating $g$ by a translation along $\im(I-A)$, we may assume that this fixed point is $0$, that is $b_0=0$. Actually, $\|b_1\|=\ell(g)$ and thus $b_1\neq0$.

Thanks to Theorem~\ref{densecc}, it suffices to show that for any $\varepsilon>0$ and $x_1,\dots,x_k\in\mathcal{H}$, one can find an elliptic element $h$, that is an element with a fixed point, such that for any $i\in\{1,\dots,k\}$, $\| g(x_i)-h(x_i)\|<\varepsilon$. Up to projecting these vectors on $\im(I-A)$ and  $\ker(I-A)$, we may assume they lie in $\im(I-A)$ or  in $\ker(I-A)$.

Let $g_0'\in\OO(\mathcal{H})$ acting like $g_0$ on the $x_i$'s in $\im(I-A)$ and being trivial on a finite codimension subspace $\mathcal{F}$ containing $\ker(I-A)$. Such an element is given by Lemma~\ref{finite rank} for $g_0$. 
 
Now choose a unit vector $u\in\mathcal{F}$ orthogonal to all the $x_i$'s and $b_1$. Fix $r$ such that all projections of the $x_i$'s on $\R b_1$ and $b_1$ have norm at most $\frac{\varepsilon r}{b_1}$. Choose $R$ large enough such that $R\left(1-\frac{1}{\sqrt{1+r^2/R^2}}\right)<\varepsilon$. Let $c$ be the point $Ru$. Let $\rho$ be the rotation with center $c$ in plane spanned by $(u,b_1)$ and angle $\alpha_1$ such that $\sin(\alpha_1)=\frac{b_1}{R}$. In the frame centered at $c$ with bases $(u,b_1)$, the point $\lambda \frac{b_1}{\|b_1\|}$ has coordinates $(R,\lambda)$. Its image by the translation of vector $b_1$ is $(R,\lambda+b_1)$ and its image by $\rho$ is $(R\cos(\alpha_1)-\lambda\sin(\alpha_1),R\sin(\alpha_1)+\lambda\cos(\alpha_1))$. With our assumptions, for $\lambda$ corresponding to the projection of one of the $x_i$'s, one has $\lambda\sin(\alpha_1)\leq \frac{\varepsilon r}{b_1}\times \frac{b_1}{R}<\varepsilon$. Since $|R-R\cos(\alpha_1)|<\varepsilon$, $| \lambda-\lambda\cos(\alpha_1)|<\varepsilon$ and $R\sin(\alpha_1)=b_1$, the images of the projections of the $x_i$'s on $\R b_1$ by the translation or the rotation are at distance at most $\sqrt{5}\, \varepsilon$.

Let us define $h$ to coincide with $g_0'$ on $\mathcal{F}^\bot$ and $\rho$ on $\mathcal{F}$. By construction, for any $x_i$, $\| g(x_i)-h(x_i)\|<\sqrt{5}\, \varepsilon$ and this finishes the proof.\end{proof}

\begin{rem} Some Polish groups, like $\mathcal{S}_\infty$, have, moreover, generic elements, that are elements with a comeager conjugacy class. One can find in \cite[Discussion 5.9]{ben2004generic} that $Isom(\mathcal{H})$ nor the unitary group have generic elements.
\end{rem}

\begin{thm}\label{nodenseconj}  The Polish group $\Isom(\H)$ has no dense conjugacy class.
\end{thm}

\begin{proof} By Corollary \ref{neutral}, if there are elements with a dense conjugacy class then those elements are neutral, that is they have vanishing translation length. These elements preserve some sphere or horosphere. 

Let $\varepsilon>0$ and $g$ be a transvection with positive translation length. Let us fix $x_1,x_2,x_3$ on the axis of $g$ such that $x_2=g(x_1)$ and $x_3=g(x_2)$. For a contradiction, assume there is a neutral element $h\in\Isom(\H)$ such that $d(h(x_i),g(x_i))<\varepsilon/2$ for $i=1,2$. So $d(h^2(x_1),x_3)<\varepsilon$. If $c$ (respectively $\xi$) is a fixed point of $h$ in $\H$ (respectively in $\partial \H$), $x_2,x_3$ are at most $\varepsilon$ from the sphere (respectively the horosphere) centered at $c$ (resp. at $\xi$) through $x_1$. Up to using a rotation around the axis of $g$, we may assume $c$ (resp. $\xi$) lies in some fixed two-dimensional totally geodesic subspace $\H^2$. 

Letting $\varepsilon \to0$ and upon extracting a converging sequence for the centers in $\overline{\H^2}$, we find a sphere or a horosphere containing $x_1,x_2,x_3$ in $\H^2$. But this impossible because in the Poincaré ball model for $\H^2$, the axis of the transvection may be represented by a straight line,  a sphere or an horosphere is a circle and the intersection of a circle and a line contains at most 2 points.  
\end{proof}

\bibliographystyle{../../Latex/Biblio/halpha}
\bibliography{../../Latex/Biblio/biblio}

\end{document}